\documentclass[10pt,a4ams]{amsart}

\usepackage{amsmath,amssymb,amsthm,hyperref}
\usepackage{enumerate}
\usepackage{todonotes}

\usepackage[english]{babel}

\usepackage{pstricks}
\usepackage{graphicx}

\usepackage{tikz}

\usepackage{subcaption}

\makeatletter
\@namedef{subjclassname@2010}{%
  \textup{2010} Mathematics Subject Classification}
\makeatother

\newtheorem{theo}{Theorem}[section]
\newtheorem{prop}[theo]{Proposition}
\newtheorem{lemm}[theo]{Lemma}
\newtheorem{coro}[theo]{Corollary}
\newtheorem{conj}[theo]{Conjecture}
\theoremstyle{remark}
\newtheorem{Deff}[theo]{Definition}
\newtheorem{exam}[theo]{Example}
\newtheorem{rema}[theo]{Remark}

\newtheorem{ques}[theo]{Problem}
\newtheorem*{clai-nn}{Claim}
\newtheorem*{theorem*}{Theorem}
\newtheorem*{corollary*}{Corollary}
\newtheorem*{rem*}{Remark}

\newcommand{\bbN}{\mathbb{N}}
\newcommand{\bbZ}{\mathbb{Z}}

\newcommand{\bbR}{\mathbb{R}}

\newcommand{\bbT}{\mathbb{T}}
\newcommand{\bbC}{\mathbb{C}}
\newcommand{\bbD}{\mathbb{D}}
\newcommand{\Pc}{\mathcal{P}}
\newcommand{\Qc}{\mathcal{Q}}

\newcommand{\Kc}{\mathcal{K}}

\newcommand{\Fc}{\mathcal{F}}
\newcommand{\Ec}{\mathcal{E}}
\newcommand{\Dc}{\mathcal{D}}

\newcommand{\Tc}{\mathcal{T}}

\newcommand{\Sc}{\mathcal{S}}
\newcommand{\Rc}{\mathcal{R}}

\numberwithin{equation}{section}
\usepackage[pagewise]{lineno}

\begin{document}

\title[Atoms of Compacta on Closed Surfaces]{Atoms of Compacta on Closed Surfaces}


\author[J. Luo]{Jun Luo}  \address{School of Mathematics\\     Sun Yat-Sen University\\ Guangzhou 512075, China} \email{luojun3@mail.sysu.edu.cn}

\author[J. Thuswaldner]{J\"org Thuswaldner} \address{Chair of Mathematics and Statistics, University of Leoben, Franz-Josef-Strasse 18, A-8700 Leoben, AUSTRIA}
\email{Joerg.Thuswaldner@unileoben.ac.at}

\author[X.T. Yao]{Xiao-Ting Yao} \address{School of Mathematics and Statistics, Guangdong University of Technology, Guangzhou 510520, China} \email{yaoxiaoting55@gdut.edu.cn}

\author[S.Q. Zhang]{Shuqin Zhang}
\address{School of Mathematics and Statistics, Zhengzhou University, 100 Science Avenue, Zhengzhou, Henan 45001, China}
\email{sqzhang@zzu.edu.cn}

\thanks{J. Luo is partly supported by National Key R\&D Program of China [No. 2024YFA 1013700]. J.~Thuswaldner is supported by the ANR-FWF Project I~6750. S.-Q. Zhang is supported by National Natural Science Foundation of China (Grant No. 12101566).}

\date{}

\begin{abstract}
For any compact set $K$ lying on a closed surface $\Sc$ we introduce a closed equivalence relation $\sim$, called the {\em Sch\"onflies equivalence} on $K$. We show that every class $[x]_\sim$ of $\sim$ is a continuum and that the resulting quotient space $K\!/\!\sim$ is a {\em Peano compactum}. By definition, all components of a Peano compactum are locally connected and for any $\varepsilon>0$ only finitely many of them have diameter greater than $\varepsilon$. The decomposition $\Dc_K=\{[x]_\sim: x\in K\}$ refines every other upper semicontinuous decomposition of $K$ into subcontinua that has a Peano compactum as its quotient space. In other words, $\Dc_K$ is the {\em core decomposition of $K$} with Peano quotient. The elements of $\Dc_K$ are called {\em atoms} of $K$. We also show that for any branched covering $f: \Sc^*\rightarrow \Sc$ from a closed surface $\Sc^*$ to $\Sc$, every atom of $f^{-1}(K)$ is sent into an atom of $K$. If $f$ is even a covering, it sends every atom of $f^{-1}(K)$ onto an atom of $K$. We illustrate our theory with examples and show that it cannot be generalized to $n$-manifolds with $n\ge 3$ by providing a detailed counterexample in~$\mathbb{R}^3$. 
\end{abstract}


\maketitle

\tableofcontents


\section{Introduction}
In the present paper we are concerned with upper semi-continuous decompositions of compact subsets of a closed surface $\mathcal{S}$ whose associated quotient spaces have certain topological properties. Recall that an {\em upper semi-continuous decomposition} ({\em usc decomposition}, for short) $\Dc$ of a topological space $X$ is a partition of $X$ into closed subsets with the following property: For each $\delta\in\Dc$ and each open set $U$ containing $\delta$ there is an open set $V$ such that $\delta\subset V\subset U$ and $V$ is a union of elements of $\Dc$. If each element of $\Dc$ is connected, we call $\Dc$ a {\em usc decomposition into subcontinua}. See for instance \cite{Daverman86,Kelley55,Nadler92} for details on usc decompositions. In the present paper we are interested in usc decompositions $\Dc$ into subcontinua with some additional topological properties that are as fine as possible.

A classical application of  such decompositions comes from Moore's decomposition theorem ({\it cf.}~\cite[Theorem~22]{Moore25}), which gives rise to usc decompositions $\Dc$ of the extended complex plane $\widehat{\mathbb{C}}$ into {\it nonseparating} subcontinua.
The resulting quotient space for such a decomposition is known to be homeomorphic to $\widehat{\mathbb{C}}$. This theorem forms the starting point of the well-developed theory of decomposition of manifolds, see {\it e.g.} \cite{Daverman86}.

We recall some basic notions and relevant background. Let $\Dc$ be an usc decomposition into subcontinua of a compactum $K$. For any $x\in K$, denote the unique element of $\Dc$ that contains $x$ by $\Dc(x)$. The map $\pi_\Dc: K\rightarrow\Dc$ sending every $x\in K$ to $\Dc(x)$ will be called the natural projection on $\Dc$. Therefore, $\pi_\Dc(x)$ has two interpretations. It can be regarded as a point in the quotient space under the quotient topology, still denoted by $\Dc$ without risk of confusion, or it can be regarded as the subcontinuum $\Dc(x)$ of $K$. If $\Dc$ is clear from the context, we just write $\pi$ instead of $\pi_\Dc$. For an usc decomposition $\Dc$ of a compact metric space $K$ (a {\em compactum}, for short), the resulting quotient space $\Dc$ is again a compactum; see for instance \cite[Theorem~20]{Kelley55} and its corollary \cite[p.~149]{Kelley55}. 

Recall that a {\em Peano continuum} is the image of the unit interval $[0,1]$ under a continuous map. Due to the Hahn-Mazurkiewicz-Sierpi\'nski Theorem, a continuum is locally connected if and only if it is a Peano continuum. The following notion of Peano compactum provides a natural generalization of Peano continuum.

\begin{Deff}[{$F$-compactum and Peano compactum; {\it cf.}~\cite[Theorem~3]{LLY-2019}}]\label{PC}
A compact metric space is called an {\em $F$-compactum} if its nondegenerate components form a {\em null sequence}, in the sense that for any $\varepsilon>0$ there are at most finitely many components of diameter greater than $\varepsilon$. A compact metric space is called   {\em Peano compactum} if it is an $F$-compactum each of whose nondegenerate components is a Peano continuum.
\end{Deff} 

\begin{Deff}[Some kinds of decompositions]\label{CD_PC}
Let $\Dc$ be an usc decomposition into subcontinua of a compact metric space $K$.
\begin{enumerate}
\item If the resulting quotient space is a Peano compactum we call  $\Dc$  a {\em Peano decomposition}.
\item If $\Dc$ is a Peano decomposition and no other Peano decomposition refines $\Dc$ then $\Dc$ is called an {\em atomic decomposition of $K$ (with Peano quotient)}.
\item If $\Dc$ is a Peano decomposition that  refines every other Peano decomposition we call it the {\em core decomposition of $K$ (with Peano quotient)} and denote it by $\Dc_K^{PC}$.
 \end{enumerate}
\end{Deff}

Here we defined the core decomposition w.r.t.\ the property of being a Peano compactum. As can be seen in \cite[Definition 1.1]{FitzGerald_Swingle67}, core decompositions can be defined in a meaningful way also w.r.t.\ other properties. For instance, the existence of the core decomposition, w.r.t.\ the property of being semi-locally connected is established in \cite[Theorem 2.7~(3)]{FitzGerald_Swingle67}. In \cite[Theorem 5.4]{FitzGerald_Swingle67}, the authors treat a special case where the core decomposition of a continuum  w.r.t.\ local connectedness exists and coincides with its core decomposition w.r.t.\ semi-local connectedness. In the present paper we will use the term ``core decomposition'' exclusively in the sense of in Definition~\ref{CD_PC}.

Let $K$ be a compactum and suppose that the core decomposition $\Dc_K^{PC}$ of $K$ exists. Then $\Dc_K^{PC}$ is the finest Peano decomposition of $K$ and the unique atomic decomposition of $K$ with Peano quotient. Each of the elements of $\Dc_K^{PC}$ is called an {\it atom} of the original compactum $K$.

For any planar compactum $K\subset{\bbC}\simeq\bbR^2$, the core decomposition $\Dc_K^{PC}$ exists by \cite[Theorem 7]{LLY-2019}.  However, for compacta $K\subset\bbR^3$, the finest  Peano decomposition may not exist. Indeed, one example of such a compactum is provided in \cite[Example 7.1]{LLY-2019}.  After some modifications of this example, we can even construct a continuum $K\subset\bbR^3$ that has uncountably many atomic decompositions; see Section~\ref{sec:ex2}. Of course, the core decomposition with Peano quotient does not exist for this continuum $K$.

\begin{rema}
It is unknown whether there exists a compactum $K$ having exactly one atomic decomposition with Peano quotient, but for that the finest Peano decomposition does not exist. This issue is related to the following question on the set $\mathfrak{M}_K$  of all the Peano decompositions of~$K$. Define a partial order $\prec$ on $\mathfrak{M}_K$  by requiring that $\Dc_1\prec\Dc_2$ if and only if $\Dc_1$ is refined by $\Dc_2$. Let $\mathfrak{M}$ be a subset of $\mathfrak{M}_K$ that is well-ordered under $\prec$, so that every two distinct elements $\Dc_1,\Dc_2\in\mathfrak{M}$ satisfy either $\Dc_1\prec\Dc_2$ or $\Dc_2\prec\Dc_1$. Let $\Dc_0$ be the coarsest usc decomposition of $K$ into subcontinua that refines every decomposition in $\mathfrak{M}$. Do we always have $\Dc_0\in\mathfrak{M}_K$?
\end{rema}

The definition of the core decomposition $\Dc_K^{PC}$ for a compactum $K\subset\widehat{\bbC}$ is motivated by two models for the topology and the dynamics of polynomial Julia sets $J$, discussed by Blokh, Curry, and Oversteegen in \cite{BCO11} and \cite{BCO13}, respectively. When $J$ is connected, the model in \cite{BCO11} gives rise to quotient spaces that are Peano continua. When $J$ is disconnected, the model in \cite{BCO13} produces quotient spaces that are finitely Suslinian. If, in addition, $J$ has no irrationally neutral cycle, those models are obtained earlier by Kiwi \cite{Kiwi04}. Notice that these models are limited to planar compacta $K$ satisfying certain conditions. Both of them  are included as special cases of the core decomposition $\Dc_K^{PC}$ obtained in \cite[Theorem 7]{LLY-2019}, which applies to all planar compacta $K$.

It turns out that atoms are maintained when mapped by branched coverings. In particular, by \cite[Theorem~1.1]{LYY-2020}, if $f:\widehat{\bbC}\rightarrow\widehat{\bbC}$ is a branched covering (such as, for instance, a rational map) and  $K \subset\widehat{\bbC}$  is a compactum, then every atom of  $f^{-1}(K)$ is sent onto an atom of $K$. In other words, we have  $f(\delta)\in\Dc_K^{PC}$ for all $\delta\in\Dc_{f^{-1}(K)}^{PC}$.  As an application of this result, we may use the atoms of a Julia set $J$ of a rational map $f$ to induce interesting factors of the dynamical system $(J,f)$.
In the special case that $f$ is a polynomial and $J$ is connected, those factors give rise to a universal approach of generalizing the (classical) core entropy for post critically finite polynomials to more general polynomials; for details we refer to \cite{LTYY-2023}.

\section{Statement of the main results}
In the present paper we focus on compacta $K$ that are contained in a {\em closed surface}, {\em i.e.}, in a compact connected $2$-manifold without boundary. Our aim is to develop a theory of core decompositions also in this setting. We will show that the core decomposition $\Dc_K^{PC}$ for a compactum $K$ contained in a closed surface $\Sc$ always exists. Moreover, for any branched covering $f$ of a closed surface $\Sc'$ onto $\Sc$ and for any atom $\delta$ of $f^{-1}(K)$ we show that $f(\delta)$ is contained in a single atom of $K$. If $f$ is even a covering map then $f(\delta)$ is an atom of $K$.

In the sequel we need the following notion. We say that $\Qc=(Q,I_1,I_2)\subset\Sc$ is a {\it marked quadrilateral} if $Q$ is homeomorphic to $[0,1]^2$ and if $I_1, I_2 \subset \partial Q$ are two disjoint closed arcs. An illustration of such a quadrilateral is provided in Figure~\ref{fig:q}.  
\begin{figure}[ht]
\includegraphics[trim = 0 15 0 45, width=0.5\textwidth]{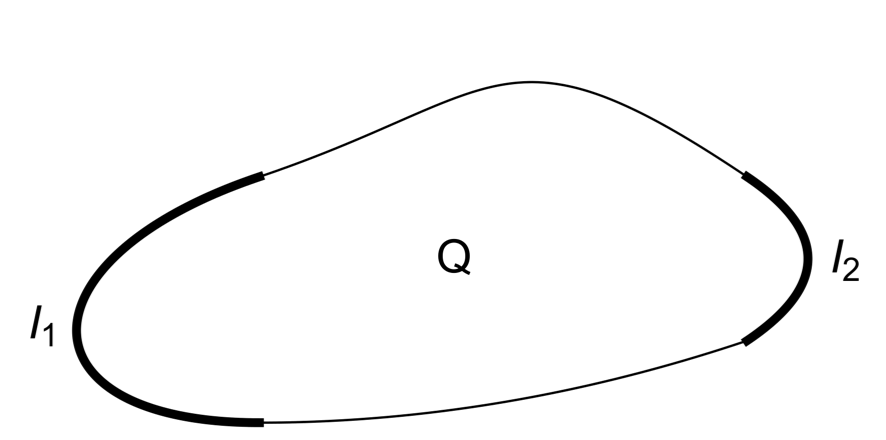}
\caption{A Quadrialteral with marked edges $I_1$ and $I_2$. \label{fig:q}}
\end{figure}
If $\Qc=(Q,I_1,I_2)\subset\Sc$ is a  marked quadrilateral with $x_1\in I_1$ and $x_2 \in I_2$ we say that $\Qc$ is a {\em marked quadrilateral with respect to the pair} $(x_1,x_2)$.

Using quadrilaterals we define the following variant of the Sch\"onflies relation from \cite[Definition 4]{LLY-2019}. This relation will be used to characterize Peano compacta on closed surfaces and to identify the core decomposition $\Dc_K^{PC}$ of any compactum $K$ on a closed surface.

\begin{Deff}[{\rm The Sch\"onflies relation on surfaces}] \label{R_K}
Let $K$ be a compactum on a closed surface $\Sc$.
The {\em Sch\"onflies relation} on $K$ is a set $R_K\subset K\times K$  consisting of all points $(x,y)\in K\times K$ such that either $x=y$ or one can find a marked quadrilateral $(Q,I_1,I_2)$ satisfying each of the following three conditions:
\begin{itemize}
\item[(a)] $x\in I_1$, $y\in I_2$ and $K\cap \partial Q \subset I_1 \cup I_2$.
\item[(b)] $K\cap Q$ has infinitely many components $Q_n$ $(n\ge1)$, satisfying $Q_n \cap I_i\not= \emptyset$ for each $i\in\{1,2\}$.
\item[(c)] $(Q_n)_{n\ge1}$ converges w.r.t.\ Hausdorff distance to a continuum $Q_\infty$ satisfying $x,y\in Q_\infty$.
\end{itemize}
\end{Deff}

When $\Sc$ is a sphere, the relation $R_K$ coincides with the Sch\"onflies relation given in \cite[Definition 4]{LLY-2019}; see Lemma \ref{lem:Rk=RkQ}.
Thus there is no harm in calling our new relation Sch\"onflies relation again. Following \cite{LLY-2019}, we also give the following definition.

\begin{Deff}[{Sch\"onflies equivalence on surfaces}]\label{SchEqSurf}
Let $K$ be a compactum on a closed surface $\Sc$. The  smallest closed equivalence relation containing $R_K$ is called  {\em Sch\"onflies equivalence}. We denote this equivalence relation by $\sim_K$ or, if  $K$ is clear from the context, just by $\sim$.
\end{Deff}

The collection $\Dc_K$ of the classes of $\sim_K$ forms a partition of $K$ into closed sets. By compactness of $K$ it is routine to verify that closedness of $R_K$ implies that $\Dc_K$ forms an usc decomposition of $K$. 

Our first main theorem shows that $\Dc_K$ is the core decomposition of $K$.

\begin{theo}\label{core}
Let $K\subset \mathcal{S}$. Then the core decomposition of $K$ exists and is given by~$\Dc_K$.
\end{theo}

The following corollary, which could also be proved directly with less effort than Theorem~\ref{core}, is an immediate consequence of Theorem~\ref{core}. Recall that,  for $x\in K$, $\Dc_K(x)$ is the unique element of $\Dc_K$ containing $x$.

\begin{coro}\label{R_K_and_PC}
Let $K\subset \mathcal{S}$. Then the following assertions are equivalent.
\begin{itemize}
\item[(1)] $K$ is a Peano compactum. 
\item[(2)] $\Dc_K(x)=\{x\}$ for all $x\in K$.
\item[(3)] The relation $R_K$ is trivial, {\it i.e.}, $R_K=\{(x,x)\colon x\in K\}$.
\end{itemize}
\end{coro}

In obtaining Theorem \ref{core}, the Sch\"onflies relation $R_K$ plays a significant role. It is used to establish the existence of the core decomposition $\Dc_K^{PC}$, and it characterizes the core decomposition. This characterization allows to construct  the elements of $\Dc_K^{PC}$ in terms of the fibers of $R_K$.  If one is just interested in the existence of $\Dc_K^{PC}$, there might be an alternative more direct approach which is outlined in Conjecture~\ref{conjecture}. However, we do not pursue this in the current paper.

The characterization of $\Dc_K^{PC}$ in terms of $R_K$ enables us to understand how the elements of $\Dc_K$ are mapped under branched coverings between closed surfaces. Given a branched covering $f:\Sc_1\rightarrow\Sc_2$ from a closed surface $\Sc_1$ onto another closed surface $\Sc_2$ and a compactum $K\subset\Sc_2$ we can establish connections between the atoms of $K$ and those of $f^{-1}(K)$. Those connections are of particular interest for the case $\Sc=\Sc_1=\Sc_2$, since we then have a dynamical system $(\Sc,f)$ such that there might exist interesting invariant domains $U$ with $U=f^{-1}(U)$. In this case, the boundary $\partial U$ is a compactum satisfying $\partial U=f^{-1}(\partial U)$. Among others, let us mention two typical choices of $f$. First, $f$ is an orientation-preserving homeomorphism of an orientable surface onto itself. See for instance \cite{Boyland94} for a survey about related works. Second, $f$ is a polynomial, considered as a self-map  of the extended complex plane onto itself. See for instance \cite{Milnor06}.

Our second main theorem reads as follows.

\begin{theo}\label{invariance}
Let $f:\Sc_1\rightarrow\Sc_2$ be a mapping from a closed surface $\Sc_1$ onto a closed surface $\Sc_2$, let $K\subset\Sc_2$  be a compactum, and let $\delta$ be an atom of
$f^{-1}(K)$.
\begin{itemize}
\item[(1)] If $f$ is a branched covering then $f(\delta)$ is a subset of an atom of $K$.
\item[(2)] If $f$ is a covering then $f(\delta)$ is an atom of $K$.
\end{itemize}
\end{theo}

Since closed surfaces may be equipped with a complex structure, there are holomorphic maps $f:\Sc_1\rightarrow\Sc_2$ between certain closed surfaces. When $\Sc_1=\Sc_2=\Sc$, so that $(\Sc,f)$ is a holomorphic dynamical system, we may consider certain compacta $K\subset\Sc$ with $f^{-1}(K)=K$ and analyze the dynamics of the factor $(K,f)$. If $K$ is not a Peano compactum, we can use the core decomposition of $\Dc_K^{PC}$ to induce a factor $\left(\Dc_K^{PC},\overline{f}\right)$, whose dynamics may be well understood. See Example~\ref{torus_2} for a holomorphic map $f$ of the 2-torus $\bbT^2$ onto itself and a continuum $K\subset\bbT^2$  with $f^{-1}(K)=K$.

The paper is organized in the following way. In Section \ref{preliminary} we recall some necessary concepts and state useful results that are known or basic. In Section~\ref{Peano} we prove Theorem~\ref{core}. The main steps of the proof are contained in Proposition~\ref{connected_fiber} which states that $\Dc_K$ is an usc decomposition of $K$ into subcontinua,  in Proposition~\ref{PC_quotient} which says that $\Dc_K$ is even a Peano decomposition of $K$, and in Proposition \ref{pro:core_1} which states that $\Dc_K$ refines every Peano decomposition of $K$. Section~\ref{PD_basics} is devoted to the proof of Theorem \ref{invariance}. In Section~\ref{examples} we give examples and remarks that illustrate our results.  In Section~\ref{sec:ex2} we analyze a concrete continuum $K_0\subset\bbR^3$ and illustrate how Theorem~\ref{core} fails for $K_0$, by showing that $K_0$ allows uncountably many atomic decompositions with Peano quotient, hence, the core decomposition of $K_0$ does not exist.
We even determine the topology of each quotient space that can result from one of the possible atomic decompositions.

\section{Preliminaries}\label{preliminary}

In this section we recall some useful concepts and results that will be needed in the proofs of our main results. Moreover,  we establish two characterizations of the Sch\"onflies relation $R_K$, given in Definition \ref{R_K}. Let us start with two well-known results from point set topology.

\begin{lemm}[{Cut Wire Theorem, see {\em e.g.}~\cite[Theorem 5.2]{Nadler92}}]\label{lem:CWT}
Let  $A_1,A_2$ be closed subsets of a compactum $X$. If no connected subset of $X$ intersects both $A_1$ and $A_2$, then $X=X_1\cup X_2$ where $X_1$ and $X_2$ are disjoint closed subsets of $X$ with $A_i\subset X_i$ for $i\in\{1,2\}$.
\end{lemm}

\begin{lemm}[{Torhorst Theorem, see {\em e.g.}~\cite[\S61, II, Theorem 4]{Kuratowski68}}]\label{lem:Torhorst}
Let  $M\subset\bbC$ be a locally connected continuum and let $R$ be a component of $\bbC\setminus M$. Then the following assertions hold.
\begin{itemize}
\item[(1)] The boundary $\partial R$ is a locally connected continuum.
\item[(2)] If $M$ has no cut point then $\partial R$ is a simple closed curve.
\end{itemize}
\end{lemm}

We also use special tilings of $\bbC$, when discussing the topology of planar comapcta.

\begin{Deff}[Brick-wall tiling]\label{brick-wall}
Let $\varepsilon>0$ and set $T_\varepsilon=\{x+y\mathbf{i}\colon 0\le x,y\le \varepsilon\}$. Call
$\displaystyle \Tc_\varepsilon=\left\{T_\varepsilon+m\varepsilon+n\left(\frac{\varepsilon}{2}+\varepsilon\textbf{i}\right):\ m,n\in\bbZ\right\} $
a  {\it brick-wall tiling} of size $\varepsilon$.
\end{Deff}
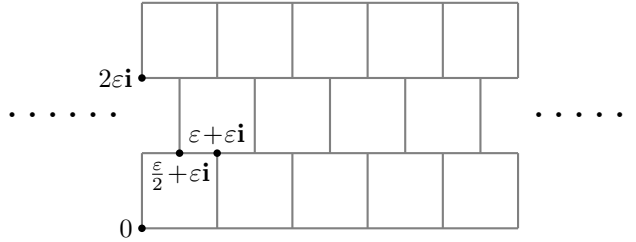
\begin{figure}[ht]
\vspace{-0.382cm}
\begin{center}
\begin{tikzpicture}[x=1cm,y=1cm,scale=1.0]
\foreach \p in {0,...,3}
{
\draw[gray,thick] (0,\p) -- (5,\p);
}
\foreach \p in {0,2}
{\foreach \q in {0,1,...,5}
{\draw[gray,thick] (\q,\p) -- (\q,\p+1);
}
}
\foreach \q in {1,...,5}
{\draw[gray,thick] (\q-0.5,1) -- (\q-0.5,2);
}
\fill[black] (0,0) circle(0.05)node[left]{$0$};
\fill[black] (0,2) circle(0.05)node[left]{$2\varepsilon\textbf{i}$};

\fill[black] (1,1) circle(0.05)node[above]{$\varepsilon\!+\!\varepsilon\textbf{i}$};
\fill[black] (0.5,1) circle(0.05)node[below]{$\frac{\varepsilon}{2}\!+\!\varepsilon\textbf{i}$};
\node[black] at (-1,1.5) {\huge$\cdots\cdots$};
\node[black] at (6,1.5) {\huge$\cdots\cdots$};
\end{tikzpicture}
\end{center}\vskip -0.25cm
\caption{A patch of the brick-wall tiling $\Tc_\varepsilon$.}\label{pic_brick-wall}
\vskip -0.25cm
\end{figure}

\begin{lemm}\label{lem:arcalphaSquare}
Let $\Qc=(Q,I_0,I_1)$ be a quadrilateral and let $K\subset Q$ be a compactum. If $K$ has $n\ge2$ components, say $P_1,\ldots, P_n$, each of which intersects both $I_0$ and $I_1$,  then there exist disjoint arcs $\alpha_1,\ldots, \alpha_{n-1}:[0,1]\to Q\setminus K$ with $\alpha_i(k)\in I_k$ $(k=0,1)$ and $\alpha_i((0,1)) \in \mathrm{int}(Q)$,
such that  $P_i$ and $P_j$ with $i\ne j$are contained in different components of $Q\setminus(\alpha_1\cup\cdots\cup\alpha_{n-1})$.
\end{lemm}

\begin{proof}
W.l.o.g.\ we may assume that $Q=[0,1]^2$ with $I_k=[0,1]\times\{k\}$ $(k\in\{0,1\})$.
Because $P_1,P_2$ are distinct components of $K$, by \cite[\S47, III, Theorem~3]{Kuratowski68} we can write $K=R_1\cup R_2$ with disjoint compact sets $R_1$ and $R_2$ satisfying $P_i\subset R_i$ $(i\in\{1,2\})$.  The brick-wall tiling $\mathcal{T}_\varepsilon$ of size $\varepsilon<\frac13{\rm dist}(R_1,R_2)$ induces a tiling of $[0,1]^2$, denoted by $\mathcal{T}_\varepsilon^\#$, each of whose elements is the intersection of $[0,1]^2$ with an element of $\mathcal{T}_\varepsilon$. Let $R_i^\#$ be the union of all the elements of $\mathcal{T}_\varepsilon^\#$ that intersect $R_i$. Let $P_1^\#$ be the component of $R_1^\#$ that contains $P_1$. Then $P_2$ is contained in the unbounded complementary component $U_1$ of $P_1^\#$. Since $P_1^\#$ is a locally connected continuum with no cut point, by Lemma~\ref{lem:Torhorst} we know that $\partial U_1$ is a Jordan curve. Let $W_1$ be the component of $[0,1]^2\setminus P_1^\#$ that contains $P_2$. Then the closure $\overline{W_1}$ is a locally connected continuum with no cut point, too. It follows that $\partial W_1$ is also a Jordan curve while the union of $\partial U_1$ and $\partial W_1$ is a $\theta$-curve \cite[p.~104]{Whyburn42}. Let $\alpha_1:[0,1]\rightarrow Q$ be an arc with $\alpha([0,1])=\partial U_1\cap\partial W_1$ such that $\alpha_1(k)\in I_k$ and $\alpha_1((0,1))\subset\left((0,1)^2\setminus K\right)$. Then $[0,1]^2\setminus\alpha_1$ consists of two Jordan domains, say $D_1$ and $D_2$, such that $\alpha_1=\partial D_1\cap\partial D_2$. W.l.o.g., we may assume that $P_i\subset\overline{D_i}$. Set $Q_i=\overline{D_i}$ and $K_i=K\cap Q_i$. Let $I_k^i=Q_i\cap I_k$ for $i\in\{1,2\}$ and $k\in\{0,1\}$. Then both $\Qc_1=(Q_1,I_0^1,I_1^1)$ and $\Qc_2=(Q_2,I_0^2, I_1^2)$ are marked quadrilaterals and $P_i$ is a component of $K_i$ for $i=1,2$. If $n=2$ we are done. Otherwise, at least one of the sets $K_1$ or $K_2$ has at least $2$ components each of which intersects both $I_0^1$ and $I_1^1$ or both $I_0^2$ and $I_1^2$,  respectively. Either way we can find an arc $\alpha_2:[0,1]\rightarrow[0,1]^2\setminus K$ with $\alpha_2(k)\in I_k$ and $\alpha_2((0,1))\subset(0,1)^2\setminus K$ that separates two of those $P_j$'s that are contained in $Q_1$ or $Q_2$. Repeating the same argument, when necessary, we will find $n-1$ arcs $\alpha_1,\ldots,\alpha_{n-1}$  that are needed.
\end{proof}

The following two lemmas are consequences of Lemma \ref{lem:arcalphaSquare}.
For the sake of convenience, hereafter in this paper we identify  $\bbR^2$ with $\bbC$ to represent the Euclidean plane. Thus $(u,v)=u+v\textbf{i}$ for all $u,v\in\bbR$ and  $I_1\times I_2=\{u+v\textbf{i}: u\in I_1, v\in I_2\}$ for any subsets $I_1,I_2\subset\bbR$.

\begin{lemm}\label{lem:arcalpha}
Let $A=\{z\in\bbC: 1\le|z|\le2\}$. If a closed set  $K\subset A$ has $n\ge 2$  components $P_1,\ldots, P_n$ each of which intersects both components of $\partial A$,  then there exist disjoint arcs $\alpha_1,\ldots, \alpha_n: [0,1]\rightarrow A\setminus K$ with $\alpha_i((0,1))\subset {\rm Int}(A)$ such that  $P_i$ and $P_j$ with $i\ne j$ are contained in different components of $A\setminus(\alpha_1\cup\cdots\cup\alpha_{n})$.
\end{lemm}
\begin{proof}
Let $I_0$ and $I_1$ be the components of $\partial A$. Because $n\ge 2$, by \cite[\S49, V, Theorem~3]{Kuratowski68} and \cite[\S57, II, Theorem~5]{Kuratowski68} we know that $K$ does not separate between $0$ and $3$.
So we can find  an arc $\beta \subset \bbC\setminus K$ connecting $0$ and $3$. Let $\alpha:[0,1]\to A\setminus K$ be a subarc of $\beta$ satisfying $\alpha(k)\in I_k$ $(k\in\{0,1\})$ and $\alpha((0,1)) \in \mathrm{int}(A)$. We can now cut $A$ along $\alpha$ (actually a tubular neighborhood of $\alpha$). By this cut, we obtain  a quadrilateral $\Qc$ with components satisfying the assumptions of Lemma~\ref{lem:arcalphaSquare}. The result now follows by applying Lemma~\ref{lem:arcalphaSquare} to this constellation.
\end{proof}

\begin{lemm}[{\bf Trapping Lemma}]\label{lem:trapping}
Let $P_0,P_1,P_2$ be distinct components of a compactum $K\subset(0,1)\times[0,1]$ satisfying the following requirements. 
\begin{itemize}
\item[(1)] $P_j\cap ((0,1)\times\{i\})\not=\emptyset$  for each $i\in\{0,1\}$ and each $j\in\{0,1,2\}$.
\item[(2)] $P_j\subset V_j$ for each $j\in\{0,1\}$, where $V_j$ is the component of $[0,1]^2\setminus P_2$ containing $\{j\}\times[0,1]$.
\item[(3)] There exist $x_i \in P_0\cap((0,1)\times\{i\})$ and $r\in(0,\frac12)$ such that $D_r(x_i)\cap P_1\ne\emptyset$ for each $i\in\{0,1\}$.
\end{itemize}
Then each component of the space $P_2\setminus\left(D_r(x_0)\cup D_r(x_1)\right)$ is a component of the space $K\setminus\left(D_r(x_0)\cup D_r(x_1)\right)$.
\end{lemm}

\begin{proof}
By condition (3) we can pick $z_0,z_1\in P_1$ with $|x_i-z_i|<r$. By Lemma \ref{lem:arcalphaSquare}, we can find two open arcs $\alpha_1,\alpha_2\subset (0,1)^2\setminus K$  each of which connects a point in $(0,1)\times\{1\}$ to one in $(0,1)\times\{0\}$  such that $(0,1)^2\setminus(\alpha_1\cup\alpha_2)$ consists of three Jordan domains $W_j$   $(0\le j\le 2)$ satisfying $P_j\subset \overline{W_j}$. See Figure \ref{trapping}, in which we represent $W_2$  by a shaded area.
\begin{figure}[ht]
     \includegraphics[trim = 0 10 0 10, width=0.618\textwidth]{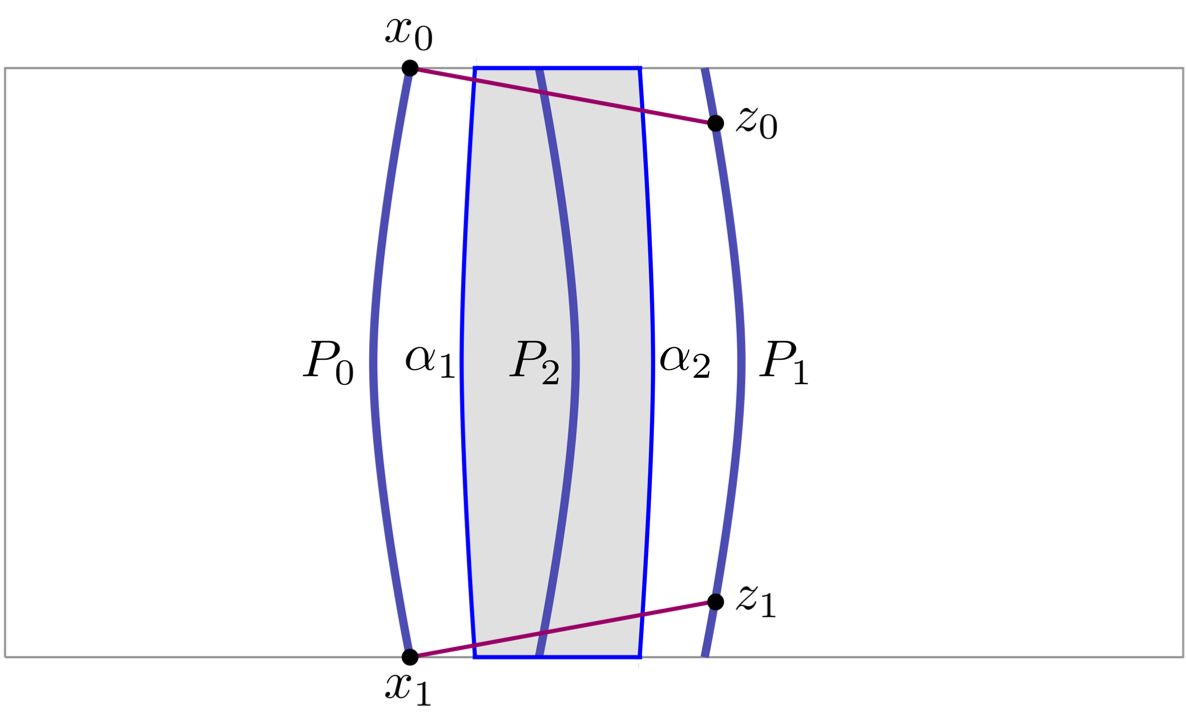} \caption{Relative locations of $x_0,x_1$  and $P_0,P_1,P_2$.\label{trapping}}
\end{figure}
Clearly, $\overline{W_2}\cap(P_0\cup P_1)=\emptyset$.
Now, for any component $R$ of $P_2\setminus\left(D_r(x_0)\cup D_r(x_1)\right)$  we can find a component $U_R$ of $(0,1)^2\setminus\left(\alpha_1\cup\alpha_2\cup\overline{z_0x_0}\cup\overline{z_1x_1}\right)$
that contains $R$.  Then $\overline{U_R}\subset(0,1)^2$ and $(K\cap \partial U_R) \subset \overline{z_0x_0}\cup\overline{z_1x_1}
\subset D_r(x_0)\cup D_r(x_1)$. It follows that $R$ is also a component of $K\setminus\left(D_r(x_0)\cup D_r(x_1)\right)$.
\end{proof}

Now we obtain two characterizations of $R_K$. The first one is the following.

\begin{lemm}\label{lem:Rk=RkQ}
Let $K$ be a compactum on a closed surface $\Sc$ and $R_K$ the Sch\"onflies relation. Then two distinct elements $x,y\in K$ are related under $R_K$ if and only if one can find
a closed annulus $A\subset S$ bounded by two closed arcs $J_1$ and $J_2$ satisfying the following three conditions.
\begin{itemize}
\item[(a)] $x\in J_1$ and $y\in J_2$.
\item[(b)]  $K\cap A$ has infinitely many components $P_n$ intersecting $J_1$ and $J_2$.
\item[(c)] $(P_n)_{n\ge1}$ converges w.r.t.\ Hausdorff distance to a continuum $P_\infty$ satisfying $x,y\in P_\infty$.
\end{itemize}
\end{lemm}
\begin{proof}
We  first deal with the ``if'' part and then the ``only if'' part.

Assume that $A$ is a closed annulus satisfying the three conditions (a), (b), and (c) in the statement of the lemma. Since $A$ is homeomorphic to the standard annulus $\{z\in\bbC: 1\le|z|\le2\}$, we may apply  Lemma~\ref{lem:arcalpha} to find disjoint arcs $\alpha_1,\alpha_2$ contained in $A\setminus K$,  each of which connects a point on $J_1$ to a point on $J_2$, such that $P_1$ and $P_2$ are contained in different components of $A\setminus(\alpha_1\cup\alpha_2)$.  Let $D_1$  be the component of $A\setminus(\alpha_1\cup\alpha_2)$ with $P_\infty\subset D_1$.
Since the arcs $\alpha_i$, except for their end points, both lie in ${\rm Int}(A)$, the triple $(\overline{D_1}, J_1\cap \overline{D_1}, J_2\cap \overline{D_1})$ is a marked quadrilateral with respect to $(x,y)$. It follows that $(x,y) \in R_K$.
\begin{figure}[ht]
\includegraphics[trim = 0 10 0 5, width=0.382\textwidth]{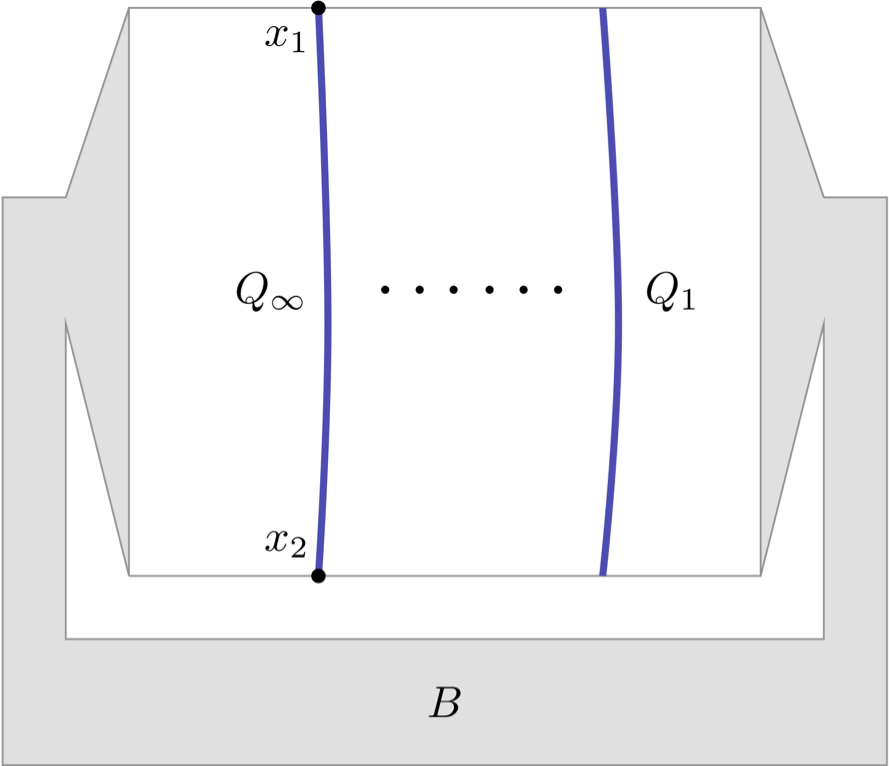}
\caption{An image of $Q$ and $B$, the latter of which is shaded.}\label{fig:RksupRkQ}
\end{figure}

Before going on to deal with the ``only if'' part, let us recall that there is a topological covering map $\varphi$ from $\bbC$ or $\widehat{\bbC}$ onto $\Sc$. Fix $(x,y)\in R_K$ and a marked quadrilateral $(Q,I_1,I_2)$ with respect to the pair $(x,y)$ that satisfies all the conditions (a), (b), and (c) in Definition \ref{R_K}. Using the covering map $\varphi$ we can find a neighborhood $V$ of $Q$ in $\Sc$ that is homeomorphic with $(0,1)^2$. By picking an appropriate embedding $h:[0,1]^2\rightarrow V$ such that $B=h\left((0,1)^2\right)$ is disjoint from $Q$ and $\overline{B}\cap Q=h(\{0,1\}\times[0,1])$, we obtain a closed annulus $A=Q\cup \overline{B}$ that satisfies each of the conditions (a), (b), and (c) in the statement of the lemma (see Figure~\ref{fig:RksupRkQ}). This completes our proof.
\end{proof}

To provide the second characterization of the relation $R_K$, we need the notion of {\em elementary region} from Whyburn \cite[p.18]{Whyburn42} and a certain notion of planarity. 

\begin{Deff}[Elementary region]\label{elementary-region}
An {\it elementary region $W(J_1,\ldots,J_k)$} is a domain in the plane that is bounded by $k\ge2$ disjoint Jordan curves $J_1,\ldots,J_k$, in a way that every $J_i$ is a component of $\partial W$.
\end{Deff}

An open annulus is just an elementary region bounded by two disjoint Jordan curves. 

\begin{Deff}[Planar and strictly planar]\label{def:SP}
 A subset of $\Sc$ is called {\em planar}, if it is homeomorphic to a subset of  $\bbC$. A continuum $K$ in a closed surface $\Sc$ is said to be {\em strictly planar} provided that one of its neighborhoods in $\Sc$ is planar. 
\end{Deff}

Notice that a Jordan curve in $\Sc$ may not be strictly planar in the above sense, since it is possible that the thickening of such a set is a M\"obius strip \cite[p.~157, Lemma (7.4)]{Armstrong}.

By \cite[Lemma 3.3]{LYY-2020}, we can generalize the characterization of $R_K$ given in Lemma~\ref{lem:Rk=RkQ} in the following way.

\begin{lemm}[{\em cf}.~{\cite[Lemma 3.3]{LYY-2020}}]\label{lem:good-R}
Let $K$ be a compactum on a closed surface $\Sc$. Then two distinct elements
$x, y\in K$ are related under $R_K$ if and only if there is an elementary region $W=W(J_1,\ldots,J_k)$, whose closure is strictly planar, that satisfies the following condition: There are two integers $i_1$ and $i_2$, such that $\overline{W}\cap K$ has infinitely many components $P_n$ intersecting both $J_{i_1}$ and $J_{i_2}$ and the limit $P_\infty=\lim\limits_{n\rightarrow\infty}P_n$ under the Hausdorff distance satisfies $x\in J_{i_1}\cap P_\infty$ and $y\in J_{i_2}\cap P_\infty$.
\end{lemm}

To conclude this section, we address how the core decomposition $\Dc_K^{PC}$ of a general compactum $K$ is connected to the core decompositions $\Dc_X^{PC}$ of its components $X$. Here we assume that all of those core decompositions $\Dc_X^{PC}$ exist. To present such a connection, let us introduce some terminology.

\begin{Deff}\label{def:F-relation}
Given a compactum $K$, let $\Fc_K$ be the finest usc decomposition of $K$ into subcontinua such that for each sequence $(P_n)_{n\in\mathbb{N}}$ of distinct components $P_n$ of $K$ that converge under Hausdorff distance its limit $P_\infty$ is contained in a single element of $\Fc_K$.
\end{Deff}

The following is immediate.
\begin{lemm}\label{lem:F_CD}
$\Fc_K$ is the finest usc decomposition such that the resulting quotient space is an $F$-compactum.
\end{lemm}

With this, we can establish the above-mentioned connection as follows.

\begin{prop}\label{prop:continuum_case}
Let $K$ be a compactum such that the core decomposition $\Dc_X^{PC}$ with Peano quotient exists for each component $X$ of $K$. Let $\Ec_K$ be the finest usc decomposition that satisfies two requirements: 
\begin{itemize}
\item[(1)] For each component $X$ of $K$ and each $\delta\in\Dc_X^{PC}$ there is a unique element $\delta'\in\Ec_K$ with $\delta\subset\delta'$.
\item[(2)] For each $\delta\in\Fc_K$ there is a unique element $\delta'\in\Ec_K$ with $\delta\subset\delta'$. 
\end{itemize}
Then $\Dc_K^{PC}$ exists and equals $\Ec_K$.
\end{prop}
\begin{proof}
Note that every component of the quotent space $\Ec_K$ is locally connected by (1). Because it is also an $F$-compactum by (2) we conclude that $\Ec_K$ a Peano decomposition. Moreover, if $\Dc$ is a Peano decomposition of $K$ and $X$ is a component of $K$, then every $\delta\in\Dc_X^{PC}$ is contained in a single element of $\Dc$. Therefore, we just need to show that every Peano decomposition $\Dc$ of $K$ is refined by $\Fc_K$. This is the case, since the projection $\pi:K\rightarrow\Dc$ sends every subcontinuum $P_\infty$ to a single point provided that $P_\infty$ is the limit under Hausdorff distance of an infinite of different components $P_n$ of $K$.
\end{proof}

\section{Existence of the core decomposition and its characterization by Sch\"onflies equivalence}\label{Peano}

This section is devoted to the proof of Theorem~\ref{core}.

Throughout this section, we will need some notation. Let $K$  be a compactum on a closed surface $\Sc$. For any point $x_0\in\Sc$ we can find a {\it coordinate chart} $\varphi: V_0\rightarrow\bbD=\{z\in\bbC: |z|<1\}$, which is a homeomorphism from a neighborhood $V_0$ of $x_0$ onto the unit disk $\bbD$ that sends $x_0$ to the origin $0$. In this way we can identify neighborhoods $D_r(x_0)=\varphi^{-1}(\{z\in\bbD: |z|<r\})$ of $x_0$ for $0<r<1$, each of which has a closure homomorphic to $\overline{\bbD}$.
Such a set $D_r(x_0)$ will be  called a \textit{disk neighborhood} of $x_0$. Moreover, if $\mathcal{R}$ is some relation, we will write $\mathcal{R}(x)=\left\{y: (x,y)\in\mathcal{R}\right\}$ for the {\em fiber} of $\mathcal{R}$ at $x$. Using this notation, we recall that $R_K(x) \subset \mathcal{D}_K(x)$ for all $x\in K$. We first establish the following result.

\begin{prop}\label{connected_fiber}
$\Dc_K$ is an usc decomposition of $K$ into subcontinua.
\end{prop}

The main part of the proof of Proposition~\ref{connected_fiber} is contained in the following lemma.

\begin{lemm}\label{connected_fiber_1}
For any $x\in K$ there is a connected set $N_x$ such that $R_K(x)\subset N_x\subset\Dc_K(x)$.
\end{lemm}

Note that we do not need to prove that $N_x$ is a continuum.

\begin{proof}
It suffices to show that for each $y\in R_K(x)$ there is a continuum $C_y$ satisfying $\{x,y\} \subset C_y \subset \Dc(x)$.
To do that, we fix a marked quadrilateral $\Qc=(Q,I_1,I_2)$ with respect to $(x,y)$, with $x\in I_1$, $y\in I_2$ and $(K\cap \partial Q) \subset (I_1 \cup I_2)$, such that $K\cap Q$ has infinitely many components $P_n$ $(n\ge1)$ satisfying $P_n \cap I_i\not= \emptyset$ for each $i\in\{1,2\}$, and $(P_n)_{n\ge1}$ converges to a continuum $P_\infty\supset\{x,y\}$ in the Hausdorff distance. 

We claim that for any $z\in P_\infty\setminus\partial Q$ and any $z_1\in P_\infty\cap\partial Q$ the closure $\overline{R_K}$ of $R_K$ contains the element $(z,z_1)$. To show this, we follow the proof of \cite[Proposition 5.1]{LLY-2019}. W.l.o.g.\ we may assume that $z_1\in I_1$ (the case $z_1\in I_2$ is symmetric). Moreover, to show that $(z,z_1)\in \overline R_K$, it suffices to verify that for any small $r>0$ there exist $z', z_1'\in P_\infty$ with $|z-z'|=|z_1-z_1'|=r$ and $(z',z_1')\in R_K$.  We represent $\Qc$ as a rectangle in Figure~\ref{location}, where the two components of $\partial Q\setminus(I_1\cup I_2)$ are denoted by $\alpha_0$ and $\beta_0$.
\begin{figure}[ht]
\vspace{-0.382cm}
\begin{center}
\begin{tikzpicture}[scale=1.3]
\footnotesize
 \pgfmathsetmacro{\xone}{0}
 \pgfmathsetmacro{\xtwo}{16}
 \pgfmathsetmacro{\yone}{0}
 \pgfmathsetmacro{\ytwo}{3.5}

\draw[gray,very thick] (\xone,\yone) -- (\xone+3*\xtwo/7,\yone) --  (\xone+3*\xtwo/7,\yone+6*\ytwo/7) --
(\xone,\yone+6*\ytwo/7) -- (\xone,\yone);

\draw[blue,very thick] (1.5*\xtwo/7,2.0*\ytwo/7) circle (0.3*\xtwo/7);
\draw[fill=blue] (1.5*\xtwo/7,2.0*\ytwo/7) circle (0.15em);
\draw[black] (1.5*\xtwo/7,2.0*\ytwo/7) node[anchor=north] {$z$};

\draw[black] (\xone,3*\ytwo/7) node[anchor=east] {$\alpha_0$};
\draw[black] (3*\xtwo/7,3*\ytwo/7) node[anchor=west] {$\beta_0$};
\draw[black] (1.75*\xtwo/7,6*\ytwo/7) node[anchor=south] {$z_1\in\gamma_1\subset I_1$};
\draw[line width=1pt,color=blue] (1.2*\xtwo/7,6*\ytwo/7) arc(180:360:0.3*\xtwo/7);

\draw[black] (1.75*\xtwo/7,0*\ytwo/7) node[anchor=north] {$\gamma_2\subset I_2$};
\draw[fill=blue] (1.5*\xtwo/7,6*\ytwo/7) circle (0.15em);

\foreach \p in {.4,.9,1.2}   \draw[gray,very thick] (0.1*\xtwo/7+\p*\xtwo/7,0*\ytwo/7) -- (0.1*\xtwo/7+\p*\xtwo/7,6*\ytwo/7);
\draw[black] (1.1*\xtwo/7,3.9*\ytwo/7) node[anchor=west] {$Q_n$};
\draw[black] (1.45*\xtwo/7,4.35*\ytwo/7) node[anchor=west] {$\eta_r$};
\end{tikzpicture}
\end{center}\vskip -0.5cm
\caption{Relative locations of $\alpha_0, \beta_0$ and $z_1\in \gamma_1$ in $Q$. \label{fig:BLpaper}}\label{location}
\vskip -0.25cm
\end{figure}
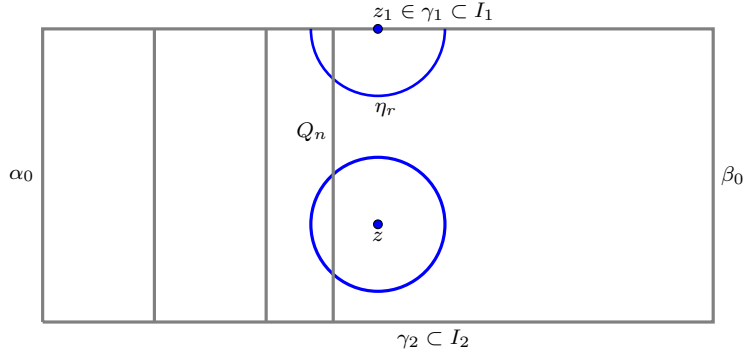
For any $r$ with  $0<r<\frac12\min\left\{{\rm dist}(z,\partial Q),\{{\rm dist} (z_1,\alpha_0\cup\beta_0)\right\}$, let $D_r(z)$ and $D_r(z_1)$ be the open  disk centered at $z$ and $z_1$ with radius $r$, respectively,  and set $\eta_r=\partial D_r(z_1)\cap \overline{Q}$. Then $\Gamma=(\partial Q\setminus D_r(z_1))\cup\eta_r$ is a simple closed curve disjoint from $\partial D_r(z)$. As $P_\infty\supset\{z,z_1\}$ and $\lim\limits_{n\rightarrow\infty}P_n=P_\infty$, we see that all but finitely many $P_n$ must intersect both $\partial D_r(z)$ and $\eta_r$ (see Figure~\ref{fig:BLpaper} for relative locations of $\eta_r$ and $\partial D_r(z)$).
Applying Lemma~\ref{lem:CWT}, we see that   $P_n\setminus\left(D_r(z)\cup D_r(z_1)\right)$ has a component $Q_n$ intersecting each of the sets $\partial D_r(z)$ and $\eta_r$. Let $A \subset Q$ be the compact subset with $\partial A=\Gamma\cup\partial D_r(z)$. Then $A$ is an annulus satisfying $P_n\setminus\left(D_r(z)\cup D_r(z_1)\right)=P_n\cap A$, which implies that  $Q_n$ is also a component of $A\cap K$.
By going to an appropriate subsequence, if necessary, we may assume that $Q_n$ $(n\ge1)$ converges to a continuum $Q_\infty$ with $Q_\infty \cap \partial D_r(z)\not=\emptyset$ and $Q_\infty \cap \eta_r \not=\emptyset$. Thus, choosing $z'\in Q_\infty \cap \partial D_r(z)$ and  $z_1'\in Q_\infty \cap \eta_r$ we have $(z',z_1')\in R_K$ by Lemma \ref{lem:Rk=RkQ}. Since $|z-z'|=|z_1-z_1'|=r$, and $r$ can be chosen arbitrarily small, this implies $(z,z_1)\in \overline{R_K}$, and the claim is proved.

Now we may set $C_y=P_\infty$. Since $z \in P_\infty \setminus \partial A$ and $z_1 \in  P_\infty \cap \partial Q$ are arbitrary, it is immediate that $\{x,y\}\subset C_y\subset\Dc(x)$.
\end{proof}

In order to complete the proof of Proposition~\ref{connected_fiber} we need the following result from continuum theory.

\begin{lemm}[{\cite[p.~278, Lemma~13.2]{Nadler92}}]\label{lem:lemNadler}
Let $X$ be a compactum and $\mathcal{D}$ an usc decomposition of $X$. Let $\mathcal{C}=\{C\subset X\colon C \ \text{is\ a\ component\ of\ some}\ D\in\mathcal{D}\}$. Then $\mathcal{C}$ is an usc decomposition of $X$ into subcontinua.
\end{lemm}

We are able to finish the proof of Proposition~\ref{connected_fiber} as follows.

\begin{proof}[{\bf Proof of Proposition~\ref{connected_fiber}}]
Let $\mathcal{C}_K=\{C\colon C \ \text{is\ a\ component\ of\ some}\ D\in\mathcal{D}_K\}$. By Lemma~\ref{lem:lemNadler}, we know that $\mathcal{C}_K$ is an usc decomposition of $K$ into subcontinua that refines $\mathcal{D}_K$.
By Lemma~\ref{connected_fiber_1}, we also know that for each $x\in K$ there is a connected set $N_x$ such that $R_K(x)\subset N_x\subset\Dc_K(x)$. Thus, by connectedness of $N_x$, $R_K(x)\subset\mathcal{C}_K(x)$ and, since $\sim_K$ is the smallest closed equivalence relation that contains $R_K$, we have $\Dc_K(x)\subset\mathcal{C}_K(x)$. This establishes the theorem.
\end{proof}

As a next step in the proof of  Theorem~\ref{core} we establish the following result.

\begin{prop}\label{PC_quotient}
$\Dc_K$ is a Peano decomposition of $K$.
\end{prop}

In the proof of this proposition, we have to make sure that certain elements of $\Dc_K$ are strictly planar in the sense of Definition~\ref{def:SP}. To assure the existence of such elements we need to prove that there do not exist too many non strictly planar elements in $\Dc_K$. In particular, we prove that even a general usc decomposition into subcontinua can contain only finitely many non strictly planar continua. In the sequel, we denote the Euler characteristic by~$\chi$. We need the following auxiliary result.

\begin{lemm}\label{lem:non-planar subcontinua}
If $\delta_0,\delta_1\subset\Sc$ are disjoint continua
and $\delta_1$ is {\em not} strictly planar, then there is a surface with boundary $\Sc_1$, whose interior contains $\delta_0$, whose closure is disjoint from $\delta_1$, and whose boundary consists of $q$ Jordan curves, such that $\chi(\Sc_1)+q>\chi(\Sc)$.
\end{lemm}
\begin{proof}
Suppose $\Sc$ is equipped with a triangulation $\Kc^0$. For $l\ge1$ recursively define $\Kc^l$ as the  {\it barycentric subdivision} of $\Kc^{l-1}$. These subdivisions will be needed to construct $\Sc_1$.
Fix $j$ such that the mesh of  $\Kc^{j}$ is smaller than one third of the distance between $\delta_0$ and $\delta_1$. Collect the $2$-simplexes of $\Kc^{j}$ intersecting $\delta_1$. These simplexes form a subcomplex, 
which provides a triangulation of  a closed subset $M_1$ of $\Sc\setminus \delta_0$.
Then, being the union of finitely many simplexes, $M_1$ has at most finitely many {\it local cut points}, say $v_1,\ldots,v_n$. Notice that each of these $v_i$  is a vertex of $\Kc^j$. Collecting all the $2$-simplexes of $\Kc^{j+1}$ that have $v_i$ as a vertex or are contained in a $2$-simplex of $\Kc^j$ lying in $M_1$, we obtain a subcomplex of $\Kc^{j+1}$ that provides a triangulation of a surface with boundary, denoted by $M_2$, which contains $M_1$.
Now consider the complement $\Sc\setminus M_2$, which has finitely many components $C_1,\ldots, C_n$. The closure of $C_i$ in $\Sc$ is a surface with boundary for each $i\in\{1,\ldots,n\}$. W.l.o.g.\ we may assume that $\delta_0 \subset C_1$. Notice that every boundary component of $\overline{C_1}$ is a one-dimensional manifold hence is a Jordan curve. Therefore, the boundary $\partial C_1$ consists of $q\ge1$ Jordan curves $J_1,\ldots,J_q$.  Thus $\Sc\setminus C_1$ is also a surface with boundary whose interior contains $\delta_1$.
Consider the disjoint union of $\Sc\setminus C_1$ and $q$ topological disks $D_1,\ldots, D_q$. Pick a homeomorphism $f_i: \partial D_i\rightarrow J_i$ for $1\le i\le q$ and form an identification space $E$ that uses those $f_i$ as attaching maps, so that every $D_i$ is glued to $\Sc\setminus C_1$ along $J_i$. See for instance \cite[p.71]{Armstrong}.
It follows that $E$ is a closed surface containing the not strictly planar set $\delta_1$. Thus $E$ is not a sphere, hence, its Euler characteristic satisfies  $\chi(E)=\chi(\Sc\setminus C_1)+q<2$. Clearly, $\chi(\partial C_1)=0$, because $\partial C_1$ is a union of disjoint Jordan curves. Since $\chi(\overline{C_1})+\chi(\Sc\setminus C_1)=\chi(\Sc) + \chi(\partial C_1) = \chi(\Sc)$,  we have
\[
\chi(\overline{C_1})+q=\chi(\Sc)-\chi(\Sc\setminus C_1)+q>\chi(\Sc)+2q-2\ge\chi(\Sc).
\]
Thus, setting $\Sc_1=\overline{C_1}$, the result is proved.
\end{proof}

We will need the following corollary of Lemma~\ref{lem:non-planar subcontinua}.

\begin{coro}[\bf Counting Lemma]\label{counting_lemma}
If $\Dc$ is an usc decomposition of $K\subset \Sc$ into subcontinua, then all but finitely many of its elements are  strictly planar.
\end{coro}

\begin{proof}
Assume on the contrary that there are infinitely many elements $\delta_k\in\Dc$, $k\ge 1$, that are not strictly planar. We may assume w.l.o.g.\ that the limit $\delta_\infty=\lim_{k\rightarrow\infty}\delta_k$  w.r.t.\ Hausdorff metric exists. Since $\Dc$ is usc, we can find $\delta_0\in\Dc$ that contains $\delta_\infty$. By Lemma~\ref{lem:non-planar subcontinua} we can find a surface with boundary $\Sc_1$ with $\Sc_1\cap\delta_1=\emptyset$ such that its interior contains $\delta_0\supset\delta_\infty$ and its boundary consists of $q\ge1$ disjoint Jordan curves. Pick $k_0\ge2$ such that $\delta_k \subset \Sc_1$ for each $k\ge k_0$. As is done in the proof of Lemma \ref{lem:non-planar subcontinua}, we may obtain a  closed surface $\Sc_1^*$ by gluing $q$ topological disks  $D_1,\ldots, D_q$ to  $\Sc_1$  along $J_i$, which then satisfies $\chi(\Sc_1)+q>\chi(\Sc)$.
This implies that $\chi(\Sc_1^*)=\chi(\Sc_1)+q > \chi(\Sc)$.
Since $\Sc_1^*$ contains $\delta_k$ with $k\ge k_0$, none of which is strictly planar, we may repeat the above construction and obtain a closed surface $\Sc_2^*$ with $\chi(\Sc_2^*)>\chi(\Sc_1^*)$ that contains infinitely many $\delta_k$. Iterating this construction indefinitely will yield an infinite sequence $(\Sc_i^*)_{i\ge 1}$ of closed surfaces with strictly increasing Euler characteristic. Since this is impossible, we have proved the result.
\end{proof}

\begin{rema}\label{2D_Geometrization}
The above proof involves Euler characteristics of certain closed surfaces and their triangulations. Using similar arguments, we can associate to every closed surface $\Sc$ a largest integer $\xi(\Sc)\ge0$ such that there is a collection of disjoint continua $N_1,\ldots,N_{\xi(\Sc)}\subset \Sc$ none of which is strictly planar. We conjecture that
\[
\xi(\Sc)=\left\{\begin{array}{ll}\frac{2-\chi(\Sc)}{2}&\text{when}\ \Sc\ \text{is\ orientable,}\\ \frac{2-\chi(\Sc)}{2} + 1&\text{when}\ \Sc\ \text{is not orientable.}\end{array}
\right.
\]
\end{rema}

\begin{proof}[{\bf Proof of Proposition~\ref{PC_quotient}}]
By Proposition~\ref{connected_fiber}, $\mathcal{D}_K$ is an usc decomposition into subcontinua. Thus it suffices to show that for any  decomposition $\Dc$ of $K$ that is not Peano, there exist  $(x_1,x_2)\in R_K$ satisfying $\Dc(x_1)\cap \Dc(x_2)=\emptyset$. 
Recall that $\Dc(x)$ for any $x\in K$ is the element of  $\Dc$ that contains $x$, which is a subcontinuum of $K$.

Let $\pi: K\rightarrow\Dc$ be the projection of $K$ onto the quotient space~$\Dc$, which is a closed, monotone map. 
As the quotient space $\Dc$ is not a Peano compactum, we can find either a component which
is not locally connected at one of its points $u_0$ or an infinite sequence of
distinct components of diameter no less than a constant $C>0$ such that their limit under
Hausdorff distance contains a point $u_0$. Either way, we can find 
a closed neighborhood $U_0$ of $u_0$ such that the component of $U_0$ containing $u_0$ is no longer a neighborhood of
$u_0$.
Then $U_0$ has infinitely many components $Q_n$ $(n\ge1)$ satisfying two requirements. First,  every $Q_n$ intersects $\partial U_0$. Second, the limit $\lim\limits_{n\rightarrow\infty}Q_n=Q_\infty$ under Hausdorff distance is a nondegenerate continuum containing $u_0$. Clearly, $Q_\infty\cap\partial U_0\ne\emptyset$ and $Q_\infty\cap {\rm Int}(U_0)$ is an uncountable set.
Moreover, $(\pi^{-1}(Q_n))_{n\ge 1}$ is an infinite sequence of pairwise disjoint subcontinua in $K$. By going to an appropriate subsequence, if necessary, we may assume that this sequence is also convergent  under Hausdorff distance.
Because $Q_\infty$ is uncountable, by  Corollary~\ref{counting_lemma}, we can find $u\in  Q_\infty$ such that $\pi^{-1}(u)$ is a strictly planar continuum. That is to say, we can find an open planar set $W\subset \Sc$ with $\pi^{-1}(u)\subset W \subset \pi^{-1}(U_0)$.
By planarity we may regard $W$ as a subset of the unit disk  $\bbD\subset\bbC$ and fix a brick-wall tiling $\Tc_\varepsilon$ of size $\varepsilon$, where $\varepsilon$ is  less than one third the distance between $\pi^{-1}(u)$ and $\partial W$.
The union of all the tiles intersecting $\pi^{-1}(u)$ is a locally connected continuum $N$ with no cut point. Notice that $\bbD\setminus N$ has finitely many components. Each of them is a Jordan domain by Lemma~\ref{lem:Torhorst}. 
By the upper semi-continuity of $\Dc$, we can find an open set $W_0$ in $\Sc$   such that $\pi^{-1}(u)\subset W_0\subset \overline{W_0}\subset {\rm Int}(N)$ and  all the elements of $\Dc$ intersecting $W_0$ are disjoint from $\partial N$. In other words, \begin{equation}\label{disjoint_atoms}
\Dc(x_1)\cap\Dc(x_2)=\emptyset\quad\text{for\ any}\ x_1\in W_0\ \text{and}\ x_2\in \partial N.
\end{equation}
Now, fix another brick-wall tiling  $\Tc_{\varepsilon'}$ of size $\varepsilon'<\varepsilon$, where $\varepsilon'$ is less than one third of the distance between  $\pi^{-1}(u)$ and $N\setminus W_0$. The union of the tiles of $\Tc_{\varepsilon'}$ intersecting $\pi^{-1}(u)$ is  a locally connected continuum $N_0$ with no cut point. Clearly, we have $N_0\subset W_0$.
Applying Lemma~\ref{lem:Torhorst} again, we see that every component of $N\setminus N_0$ is bounded by a Jordan curve. It follows that  ${\rm Int}(N)\setminus N_0$ is an elementary region in the sense of Definition \ref{elementary-region}.
Recall that $\big\{\pi^{-1}(Q_n)\big\}_{n\ge 1}$ is convergent under the Hausdorff distance and that its limit intersects $\pi^{-1}(u)$. It follows that all but finitely many of the continua $\pi^{-1}(Q_n)$ intersect both $\partial N_0$ and $\partial N$.
For any of those $n\ge1$, we can use Lemma~\ref{lem:CWT} to find a component $P_n$ of $\pi^{-1}(Q_n) \cap (N\setminus{\rm Int}(N_0))$ that intersects $\partial N$ and $\partial N_0$ both. Because $\partial (N\setminus N_0)=\partial N \cup\partial N_0$ is a separation and $\partial N$ as well as $\partial N_0$ are finite unions of disjoint Jordan curves, we can find Jordan curves $\gamma_1 \subset \partial N_0$ and $\gamma_2 \subset \partial N$ such that there are infinitely many $P_n$ that intersect $\gamma_1$ and $\gamma_2$ both. Therefore, by Lemma \ref{lem:good-R} we can further find $x_1 \in \gamma_1\subset \partial N_0\subset W_0$ and $x_2 \in \gamma_2\subset\partial N$ with $(x_1,x_2)\in R_K$. By equation \eqref{disjoint_atoms}, this  proves the result.
\end{proof}

To obtain Theorem~\ref{core}, it suffices to combine Proposition \ref{PC_quotient} with the following result.

\begin{prop}\label{pro:core_1}
$\Dc_K$ refines every Peano decomposition of $K$.
\end{prop}

\begin{proof}
It suffices to show that for any usc decomposition $\Dc$ of $K$ into subcontinua, if there exists $(x_1,x_2)\in R_K$ such that $\Dc(x_1)\cap \Dc(x_2)=\emptyset$, then the quotient space $\Dc$ is not a Peano compactum. Pick a marked quadrilateral $\Qc=(Q,I_1,I_2)\subset \Sc$ with respect to $(x_1,x_2)$, such that
\begin{itemize}
\item[(a)] $x_1$ and $x_2$ belong to the two marked sides $I_1$ and $I_2$ respectively;
\item[(b)] $K\cap Q$ has infinitely many components $Q_n$ $(n\ge1)$ each of which intersects both of the arcs $I_1$ and $I_2$;
\item[(c)] $(Q_n)_{n\ge1}$ converges in Hausdorff distance to a continuum $Q_\infty\supset\{x_1,x_2\}$.
\end{itemize}
Let $\pi: K\rightarrow\Dc$ be the projection onto the quotient space. Because $\pi(x_1)\neq \pi(x_2)$, the continuum $\pi(Q_\infty)$ is nondegenerate and, hence, the diameters of $\pi(Q_{n})$ are bounded away from zero. Thus one of the following two cases must occur. 

\smallskip
\noindent {\em Case 1:}  There is an infinite sequence $(Y_i)_{i\ge 1}$ of pairwise disjoint components of the quotient space of $\Dc$ and a strictly increasing sequence $(n_i)_{i\ge 1}$ of integers satisfying $\pi(Q_{n_i}) \subset Y_i$. Thus, the diameters of $Y_i$ are bounded away from zero, hence, the quotient space of $\Dc$ is not a Peano compactum and we are done.

\smallskip
\noindent{\em Case 2:} There exist $N\ge 1$ and a component $Y_0$ of the quotient space of $\Dc$ 
such that $\pi(Q_n) \subset Y_0$ holds for each $n\ge N$ (by taking a subsequence we may assume w.l.o.g.\ that $N=1$.). In this case we also have $\pi^{-1}(Y_0)\supset Q_\infty$. To prove that $\Dc$ is not a Peano compactum it suffices to prove that $Y_0$ is not locally connected. Pick two open balls $B_j$ centered at $\pi(x_j)$ (with $j=1,2$) such that $\overline{B_1}\cap \overline{B_2} =\emptyset$. See Figure \ref{fig:finest_proof}. \begin{figure}[ht]
         \centering
     \includegraphics[trim = 0 10 0 10, width=0.618\textwidth]{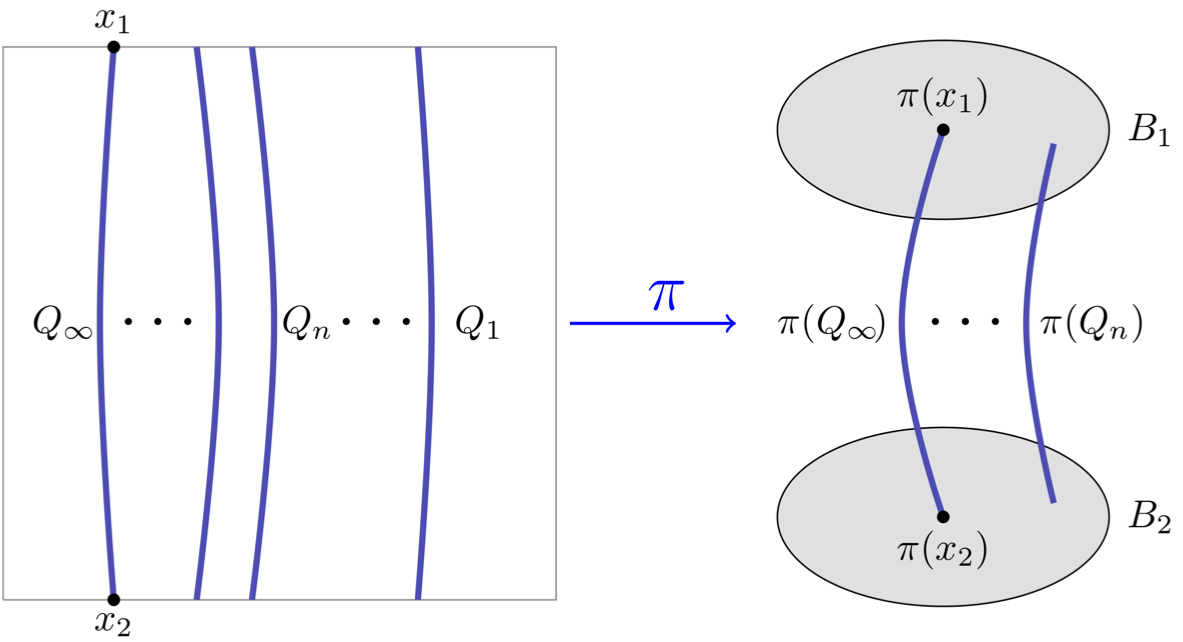}

     \caption{The components $Q_n$ and their images under $\pi$.
     }\label{fig:finest_proof}
\end{figure}
Set $L=\pi^{-1}(Y_0)$ and $L_0=L\setminus\big(\pi^{-1}(B_1)\cup \pi^{-1}(B_2)\big)$. By definition this implies that $\pi(L)=Y_0$ and $\pi(L_0)=Y_0\setminus(B_1\cup B_2)$. Note that $\pi(L_0)$ has nonempty interior in $Y_0$, whose relative topology is inherited from the quotient space $\Dc$. Moreover, its boundary in $Y_0$ is a subset of $\partial B_1 \cup \partial B_2$. Also not that $\pi(Q_\infty)$ is
connected and intersects ${B_1}$ as well as  ${B_2}$. By condition (c) we may assume that every $\pi(Q_n)\ (n\ge1)$ intersects both ${B_1}$ and  ${B_2}$. Since every $\pi(Q_n)$ $(n\ge1)$ is connected, for each $n\ge 1$ we can pick  $y_n\in Q_n$  such that the distance from $\pi(y_n)$ to $\overline{B_1}$ and that from $\pi(y_n)$ to $\overline{B_2}$ are equal. Going to an appropriate subsequence, we may assume that $y_n\rightarrow y_\infty$. Then
$\pi(y_\infty)$ belongs to $\pi(Q_\infty)\cap{\rm Int}\left(\pi(L_0)\right)$. To complete our proof, we show that the component of $\pi(L_0)$ containing $\pi(y_\infty)$ is not a neighborhood of $\pi(y_\infty)$, with respect to the relative topology of $Y_0$.  To do that, we pick an open neighborhood of $Q$, say $Q^*$, which may be regarded as a subset on the plane. Fix two open disks $D_i$ centered at $x_i$ $(i\in\{1,2\})$, such that $\overline{D_1}\cap\overline{D_2}=\emptyset$ and $D_i\subset Q^*\cap\pi^{-1}(B_i)$ for $i\in\{1,2\}$. Since $\lim\limits_{n\rightarrow\infty}Q_n=Q_\infty\supset\{x_1,x_2\}$, we may assume that every $Q_n$ intersects both $D_1$ and $D_2$. By condition (b) and the containment $\{y_n: n\ge1\}\subset L_0$ we may apply Lemma \ref{lem:trapping} to exhibit an infinite set $\bbN_0\subset\bbN$ such that for any $n\in\bbN_0$  the component of $Q_n\setminus(D_1\cup D_2)$  containing $y_n$ is also a component of $K\setminus(D_1\cup D_2)$.  Thus no two of  $\pi(y_n)$ with $n\in\bbN_0$ lie in the same component of $\pi(L_0)=Y_0\setminus(B_1\cup B_2)$. This proves the result, since $\pi(y_n)\rightarrow\pi(y_\infty)$. 
\end{proof}

\begin{rema}
Proposition \ref{pro:core_1} is also the consequence of Propositions \ref{closure_R} and \ref{prop:R_K_tilde}. Because the proof of Proposition \ref{pro:core_1} is more direct, we kept it for the sake of self-containment.
\end{rema}

\section{How atoms are mapped under branched coverings}\label{PD_basics}
This section is devoted to the proof of Theorem \ref{invariance}. We first introduce a  new relation.

\begin{Deff}[The relation $\widetilde{R_K}$]\label{def:RTilde}
The relation  $\widetilde{R_K}$ on an arbitrary compactum $K$ (not necessarily embedded in a surface) consists of all points $(x,x)\in K\times K$ and all points $(x,y)$ with $x\ne y$ such that both $x$ and $y$ possess a decreasing sequence of neighborhoods $(U_{x,n})_{n\ge 1}$ and $(U_{y,n})_{n\ge 1}$  satisfying the following two conditions.
\begin{itemize}
\item[(a)]  $\bigcap_{n\ge 1}U_{x,n}=\{x\}$ and $\bigcap_{n\ge 1}U_{y,n}=\{y\}$.
\item[(b)] For each $n\ge1$ the set $K\setminus\left(U_{x,n}\cup U_{y,n}\right)$ has infinitely many components $Q_{k,n}$ ($k\ge 1$) intersecting both of the boundaries $\partial U_{x,n}$ and $\partial U_{y,n}$.
\end{itemize}
\end{Deff}

By definition, $\widetilde{R_K}$ is reflexive and symmetric. Moreover, we can easily check that  $\widetilde{R_K}$ is a  closed relation.
The relation $\widetilde{R_K}$ has a close connection with the Peano decompositions of  $K$. Indeed, we have the following result.

\begin{prop}\label{closure_R}
If $\Dc$ is a Peano decomposition of a compactum $K$ and $x\in K$, then $\widetilde{R_K}(x)\subset \Dc(x)$.
\end{prop}
\begin{proof}
Let $\pi: K\rightarrow\Dc$ be the monotone projection of $K$ onto the quotient space of $\Dc$. We will provide a proof by contradiction.  Suppose on the contrary that there exists  $(x,y)\in\widetilde{R_K}$ with $\pi(x)\ne\pi(y)$. In the quotient space $\Dc$ pick two open sets $V_x$ and $V_y$ with $\overline{V_x}\cap \overline{V_y}=\emptyset$ and $\pi(x)\in V_x$ and $\pi(y)\in V_y$.
As $(x,y)\in\widetilde{R_K}$, we can find sequences $(U_{x,n})_{n\ge 1}$ and $(U_{y,n})_{n\ge 1}$ satisfying conditions (a) and (b) of Definition~\ref{def:RTilde}. Denote the infinitely many components intersecting both of the boundaries $\partial U_{x,n}$ and $\partial U_{y,n}$ by $Q_{k,n}$ ($n,k\ge 1$). Fix $n\ge1$ such that $U_{x,n}\subset\pi^{-1}(V_x)$ and $U_{y,n}\subset\pi^{-1}(V_y)$  hold.
Using Lemma \ref{lem:CWT}, we see that every $Q_{k,n}$ $(k\ge1)$ contains a subcontinuum $P_k$ which is a component of $Q_{k,n}\setminus\left(\pi^{-1}(V_x)\cup\pi^{-1}(V_y)\right)$ and intersects $\partial \pi^{-1}(V_x)$ and $\partial \pi^{-1}(V_y)$. By Lemma~\ref{lem:trapping}, each of these $P_k$ is also a component of
\[
K\setminus\Big(\pi^{-1}(V_x)\cup\pi^{-1}(V_y)\Big).
\]
As $\pi: K\rightarrow\Dc$ is a monotone map, we see that every $\pi(P_k)$ is necessarily a component of $\Dc\setminus(V_x\cup V_y)$, which intersects $\partial V_x$ and $\partial V_y$ both. Since the quotient space $\Dc$ is a Peano compactum, it has a component $M$ that contains all but finitely many of the above continua $\pi(P_k)$.
By going to an appropriate subsequence, we may assume that $\pi(P_k)$ converge under the Hausdorff distance to a continuum $P_\infty^*$.
Pick $\xi\in P_\infty^*$ with ${\rm dist}(\xi,\partial V_x)={\rm dist}(\xi,\partial V_y)$. Then $\Dc\setminus(V_x\cup V_y)$ is a neighborhood of $\xi$ and its component containing $\xi$ is no longer a neighborhood of $\xi$. Thus $M$ is not locally connected at $\xi$.
This is absurd, since the quotient space $\Dc$ is assumed to be a Peano compactum thus all of its components, including $M$, must be locally connected.
\end{proof}

The following result is also useful.
\begin{prop}\label{prop:R_K_tilde}
If $K$ lies on a closed surface $\Sc$, then $\widetilde{R_K}$ contains $R_K$.
\end{prop}
\begin{proof}Fix a metric $\rho$ on $\Sc$ that is compatible with the topology of $\Sc$.
Given $(x_1,x_2)\in R_K$, we can fix  a marked quadrilateral $\Qc=(Q,I_1,I_2)\subset \Sc$ with respect to $(x_1,x_2)$, such that
\begin{itemize}
\item[(a)] $x_1\in I_1$, $x_2\in I_2$, while $\partial Q\setminus(I_1\cup I_2)$ is disjoint from $K$;
\item[(b)] $K\cap Q$ has infinitely many components $Q_n$ $(n\ge1)$ each of which intersects $I_1$ and $I_2$;
\item[(c)] $Q_n$ $(n\ge1)$ converge under Hausdorff distance to a continuum $Q_\infty\supset\{x_1,x_2\}$.
\end{itemize}To show that $(x_1,x_2)\in\widetilde{R_K}$, we may consider $Q$ to be $[0,1]^2$ with $I_1=[0,1]\times\{1\}$ and $I_2=[0,1]\times\{0\}$.  Let $Q_0$ be the component of $K\cap Q$ that contains $Q_\infty$. W.l.o.g., we assume that all $Q_n$ $(n\ge1)$ lie in the same component of $Q\setminus Q_0$.
For any $k\ge1$, let $D_{r_k}(x_i)$ $(i=1,2)$ be the disk neighborhoods of $x_i$ with radius $r_k=\frac{1}{k+2}\rho(x_1,x_2)$. By Lemma \ref{lem:trapping}, we can find a Jordan domain $U_n\subset Q^\circ$ and a continuum $P_n\subset Q_n$, for infinitely many $n\ge1$, such that the following conditions are all satisfied. First, $P_n\subset U_n$. Second, $P_n$ is a component of $(K\cap Q)\cup(\Sc\setminus {\rm Int}(Q))$. Third, $P_n$ intersects $\partial D_{r_k}(x_1)$ and  $\partial D_{r_k}(x_2)$.
By the definition of $\widetilde{R_K}$ and because $k\ge1$ was arbitrary, this proves the result.
\end{proof}

Notice that when $K$ is actually planar, so that it may be embedded into $\widehat{\bbC}$, we can infer from \cite[Theorem 1.3]{LYY-2020} that $\widetilde{R_K}=\overline{R_K}$. In such a case, we know that every fiber $\widetilde{R_K}(x)$ is a continuum \cite[Theorem 1.4]{LYY-2020}. When $K$ is nonplanar and lies on a closed surface, the following problem remains open.

\begin{ques}
Is $\widetilde{R_K}(x)$ connected for each $x\in K$? Is it true that $\widetilde{R_K}=\overline{R_K}$?
\end{ques}

\begin{rema}
Fix a general compactum $K$.
By Proposition \ref{closure_R}, if the core decomposition $\Dc_K^{PC}$ exists then every fiber $\widetilde{R_K}(x)$ with $x\in K$ lies in a single element of $\Dc_K^{PC}$. In other words, $\Dc_K^{PC}$ splits no fiber of $\widetilde{R_K}$.  However, it is not clear whether there exists an usc decomposition into subcontinua, whose quotient space is not a Peano compactum, that refines $\Dc_K^{PC}$ and splits none of the fibers $\widetilde{R_K}(x)$.
Note that if further $K$ belongs to a closed surface then $\overline{R_K}\subset\widetilde{R_K}$ by Proposition~\ref{prop:R_K_tilde} and, by Proposition~\ref{closure_R}, we can easily infer Proposition~\ref{pro:core_1}. Therefore, by Proposition~\ref{PC_quotient} we further see that $\Dc_K^{PC}$ is the  finest usc decomposition  that splits no fiber of $\widetilde{R_K}$. In other words, the smallest closed equivalence containing $\widetilde{R_K}$ coincides with $\sim_K$. Notice that in such a situation the classes of $\sim_K$ are all connected thus $\Dc_K^{PC}$ is also the finest usc decomposition of  $K$ into subcontinua that splits no fiber of $\widetilde{R_K}$.
\end{rema}


In view of Propositions \ref{closure_R} and \ref{prop:R_K_tilde}, the first part of Theorem~\ref{invariance} is implied by the following result.

\begin{lemm}\label{into}
Let $f:\Sc_1\rightarrow\Sc_2$ be a branched covering map of a closed surface $\Sc_1$ onto a closed surface $\Sc_2$, $K\subset\Sc_2$ a compactum and $L=f^{-1}(K)$.  Under the above assumptions, every fiber of $\widetilde{R_L}$ is sent into a fiber of $\widetilde{R_K}$. In particular, if $K$ is a Peano comapctum so is $L$.
\end{lemm}

\begin{proof}
Given $(x_1,x_2)\in\widetilde{R_L}$ with  $f(x_1)\ne f(x_2)$, we will show that $(f(x_1),f(x_2))\in\widetilde{R_K}$.
Choose two disk neighborhoods  $D_i$ in $K$, centered at $f(x_i)$ and with small radius $r>0$, such that $\overline{D_1}\cap\overline{D_2}=\emptyset$. Fix open sets $U_i$ with $x_i \in U_i$ ($i\in\{1,2\}$) with $f(U_i)\subset D_i$ such that  $L\setminus\left(U_1\cup U_2\right)$ has infinitely many components, say $\{Q_k: k\ge1\}$, each of which intersects both of the boundaries $\partial U_1$ and $\partial U_2$. By Lemma \ref{lem:CWT} every $Q_k$ $(k\ge1)$ contains a subcontinuum $P_k$ which is a component of $Q_k\setminus\left(f^{-1}(D_1)\cup f^{-1}(D_2)\right)$ and intersects both $\partial f^{-1}(D_1)$ and $\partial f^{-1}(D_2)$. Each of these $P_k$ is also a component of $L\setminus\left(f^{-1}(D_1)\cup f^{-1}(D_2)\right)$.
Let $M_k$ be the component of $K\setminus(D_1\cup D_2)$ that contains $f(P_k)$. As $f$ is finite-to-one, the preimage $f^{-1}(M_k)$ has finitely many components. Each of them is sent by the open map $f$ onto $M_k$, since we can apply \cite[p.~284, Theorem 13.14]{Nadler92}.
It follows that $P_k$ is a component of $f^{-1}(M_k)$ and satisfies $M_k=f(P_k)$. That is to say, every $f(P_k)$ is a component of $K\setminus(D_1\cup D_2)$ and intersects both $\partial D_1$ and $\partial D_2$. Because $r>0$ was arbitrary, it is immediate that $(f(x_1),f(x_2))\in\widetilde{R_K}$.
\end{proof}

\begin{proof}[Proof of Theorem~\ref{invariance}~(1)]
Let $L=f^{-1}(K)$. Then $\mathcal{D}_L^*=\left\{f^{-1}(\delta): \delta\in\mathcal{D}_K^{PC}\right\}$ is a usc decomposition of $L$. By Lemma 5.6, no fiber of $\widetilde{R_L}$ is split by $\mathcal{D}_L^*$, so that each fibers of
$\widetilde{R_L}$ is contained in a single element of $\mathcal{D}_L^*$. 
By definition, $\mathcal{D}_L$ is the finest usc decomposition which splits none of the fibers of $R_L$.
Theorem~\ref{core} implies that the core decomposition $\mathcal{D}_L^{PC}$ exists and coincides with $\mathcal{D}_L$. Thus also $\mathcal{D}_L^{PC}$ splits no fibers of $R_L$. Since $\mathcal{D}_L^{PC}$ is a Peano decomposition, Propositions 5.2 implies that it splits none of the fibers of $\widetilde{R_L}$. 

Let $\widetilde{\mathcal{F}}$ be the family of all usc decompositions that split none of the fibers of $\widetilde{R_L}$ and let ${\mathcal{F}}$ be the family of all usc decompositions that split none of the fibers of $R_L$. Then by Proposition~5.3 we get $\widetilde{\mathcal{F}}\subset\mathcal{F}$. Because $\mathcal{D}_L^{PC}$ is the finest member of $\mathcal{F}$, {\em a fortiori}, it has to be the finest member of $\widetilde{\mathcal{F}}$. In other words, the core decomposition $\mathcal{D}_L^{PC}$ is the finest usc decomposition that splits none of the fibers of $\widetilde{R_L}$. It follows that $\mathcal{D}_L^{PC}$ refines $\mathcal{D}_L^*$, which then implies that every element of $\mathcal{D}_L^{PC}$ is sent by $f$ into a single element of $\mathcal{D}_K^{PC}$. This proves  Theorem 2.5~(1).
\end{proof}

To prove Theorem~\ref{invariance}~(2) we now consider covering maps between closed surfaces and obtain the following lemma.

\begin{lemm}\label{lem:fiber_onto}
Let $f:\Sc_1\rightarrow\Sc_2$ be a covering map of a closed surface $\Sc_1$ onto a closed surface $\Sc_2$, $K\subset\Sc_2$ a compactum and $L=f^{-1}(K)$.  If $x\in K$ and $u\in f^{-1}(x)$ then for any $y\in R_K(x)$ there exists $v\in R_L(u)$ with $f(v)=y$.
\end{lemm}

\begin{proof}
Fix a marked quadrilateral $Q\subset\Sc_2$ with respect to $(x,y)$ that satisfies conditions (a), (b), (c) of
Definition~\ref{R_K}. The preimage $f^{-1}(Q)$ has finitely many components, each of which is also a quadrilateral. Moreover, every of those components contains a point in $f^{-1}(x)$. Pick the one containing $u$ and denote it by $Q_u$. Since $f$ restricted to $Q_u$ is a homeomorphism onto $Q$, one can find $v\in Q_u$ with $f(v)=y$ such that $Q_u$ is a marked quadrilateral with respect to $(u,v)$. It follows that $(u,v)\in R_L$ that satisfies conditions (a), (b), (c) of Definition~\ref{R_K}. This proves the result.
\end{proof}

With this, the second part of Theorem \ref{invariance} relies on the next two lemmas.

\begin{lemm}\label{lem:fiber_onto_closure}
Let $f:X\rightarrow Y$ be an open map of a compactum onto another and $\Rc_X,\Rc_Y$  reflexive symmetric relations on $X$ and $Y$, respectively. If $f\left(\Rc_X(x)\right)\supset\Rc_Y(f(x))$  for all $x\in X$ then  $f\left(\overline{R_X}(x)\right)\supset\overline{R_Y}(f(x))$  for all $x$.
\end{lemm}
\begin{proof}
For any $x\in X$ and any $v\in\overline{\Rc_Y}(u)$, where $u=f(x)$, we can find an infinite sequence $(u_k,v_k)_{k\ge1}$ in $\Rc_Y$ with $u_k\rightarrow u$ and $v_k\rightarrow v$. By the openness of $f$, for all $k\ge1$ we can find $x_k\in f^{-1}(u_k)$ with $x_k\rightarrow x$. Since $f\left(\Rc_X(x)\right)\supset\Rc_Y(f(x))$ by assumption, for all $k\ge1$ we can further find $y_k\in\Rc_X(x_k)$ such that $f(y_k)=v_k$. This ensures that $f(y)=\lim\limits_{k\rightarrow\infty}f(y_k)=v$.
\end{proof}

\begin{lemm}\label{lem:LYY_invariance_lemma}
Let $f:X\rightarrow Y$ be an open map of a compactum onto another. 
Let $\mathcal{R}_X$ on $X$ and $\mathcal{R}_Y$ on $Y$ be two closed reflexive symmetric relations such that $f(\mathcal{R}_X(x))\supset \mathcal{R}_Y(f(x))$ for all $x\in X$. If $\sim_X$ (respectively, $\sim_Y$) is the minimal closed equivalence relation containing $\mathcal{R}_X$ (respectively, $\sim_Y$), then  $f\left([x]_{\sim_X}\right)\supset [f(x)]_{\sim_Y}$  for all $x\in X$.	
\end{lemm}

\begin{rema} Lemma~\ref{lem:LYY_invariance_lemma} has its origin in \cite[Lemma 3.5]{LYY-2020}, although our assumptions here are slightly weaker. For the sake of completeness, we give a self-contained proof below.	
\end{rema}

\begin{proof}
First, we define an equivalence  $\approx$ on $Y$ such that  $y_1, y_2\in Y$ are related under $\approx$ if and only if the union of the classes of 
$\sim_X$ intersecting $f^{-1}(y_1)$ equals the union of those intersecting $f^{-1}(y_2)$. 
We claim that 
\begin{equation}\label{eq:simapprox}
f([x_0]_{\sim_X}) \supset [f(x_0)]_\approx. 
\end{equation}
Indeed, suppose $y\in [f(x_0)]_\approx$, \emph{i.e.}, $y\approx f(x_0)$. Then we have
\[
\bigcup_{f(x)=f(x_0)}[x]_{\sim_X}= \bigcup_{f(x')=y}[x']_{\sim_X},
\]
and there is a one-to-one correspondence between the classes in the left and right hand side unions. In particular, the class $[x_0]_{\sim_X}$ occurs in the union on the right hand side. In other words, there is $x'$ with $f(x')=y$ such that $[x_0]_{\sim_X}=[x']_{\sim X}$. Thus $y\in f([x_0]_{\sim_X})$ and the claim is proved.

Second, we verify that $\approx$ contains $\mathcal{R}_Y$ as subsets of $Y\times Y$. To do that, by symmetry of $\mathcal{R}_Y$ it  suffices to show  that  for any $y_1,y_2\in Y$ that are related under $\mathcal{R}_Y$ we have
\begin{align}\label{eq:twounions}
	\bigcup_{f(x)=y_1}[x]_{\sim_X}&\subset\ \bigcup_{f(x')=y_2}[x']_{\sim_X}.
\end{align}
By assumption, for any $x$ with $f(x)=y_1$ we shall have  $f(\mathcal{R}_X(x))\supset \mathcal{R}_Y(y_1)$. Since $y_2\in\mathcal{R}_Y(y_1)$, this yields that  $y_2\in f(\mathcal{R}_X(x))$ and, hence, we can find $x'\in\mathcal{R}_X(x)$ such that $f(x')=y_2$. Moreover, by the definition of $\sim_X$, $x'\in \mathcal{R}(x)$ implies that $[x]_{\sim_X}=[x']_{\sim_X}$. Thus each class that is contained in the union in the left hand side of \eqref{eq:twounions} is contained in the union in the right hand side of  \eqref{eq:twounions} as well. This proves~\eqref{eq:twounions}.

Third, we claim that $\approx$ is a closed equivalence. 
Since $\approx$ is considered as a subset of the product space $Y\times Y$, we just need to verify that the complement of $\approx$ in $Y\times Y$ is open. Indeed, if $y_0,y_1\in Y$ are not related under $\approx$  then w.l.o.g.\ we may assume that there is a point $x_0\in f^{-1}(y_0)$ such that the class $[x_0]_{\sim_X}$ is disjoint from $f^{-1}(y_1)$. Now we can apply the upper semi-continuity of the decomposition $\left\{[x]_{\sim_X}: x\in X\right\}$ and find two disjoint open sets $U$ and $V$ which are saturated with respect to $\sim_X$, such that $U$ contains $x_0$ and $V$ contains $f^{-1}(y_{1})$. 
The openness of $f$ then ensures that $f(U)$ and $f(V)$  are disjoint open sets such that $U\times V$ is disjoint from $\approx$. This proves the claim.

Now, using that $\sim_Y$ is the minimal equivalence relation containing $\mathcal{R}_Y$ and  \eqref{eq:simapprox} we gain
\[
[f(x)]_{\sim_Y} \subset
  [f(x)]_{\approx}  \subset
f\left([x]_{\sim_X}\right),
\]
and the result follows.
\end{proof}

\begin{proof}[Proof of Theorem~\ref{invariance}~(2)]
Let $f:\Sc_1\rightarrow\Sc_2$ be a covering from a closed surface $\Sc_1$ onto a closed surface $\Sc_2$ and let $K\subset\Sc_2$  be a compactum. Since we already proved Theorem~\ref{invariance}~(1), we only need to prove that for any  atom $\delta$ of $L=f^{-1}(K)$ and any $x\in\delta$ the atom of $K$ containing $f(x)$ is contained in $f(\delta)$. 

To prove this, first observe that by Lemma~\ref{lem:fiber_onto} we have $f(R_L(x)) \supset R_K(f(x))$. By Lemma~\ref{lem:fiber_onto_closure} we even get that $f\big(\overline{R_L}(x)\big) \supset \overline{R_K}(f(x))$. From this, Lemma~\ref{lem:LYY_invariance_lemma} finally yields that $f(\delta)=f([x]_{\sim L})$ contains $[f(x)]_{\sim K}$, which is the atom of $K$ containing $f(x)$.  
\end{proof}

\section{Examples for determining the core decomposition}\label{examples}

In this section we consider special compacta $K$ on a closed surface $\Sc$ and determine their atoms. Given such a compactum $K$, fix a polygon $\Pc\subset\bbC$ and a rule that identifies certain pairs of the sides of $\Pc$ such that the resulting quotient space is topologically equivalent to $\Sc$. Let $\tau: \Pc\rightarrow\Sc$ be the underlying projection and $L=\tau^{-1}(K)$. Let  $\Dc_K^{PC}$ and $\Dc_L^{PC}$ be the core decompositions of $K$ and $L$, respectively. In special cases, we have exact information on how the atoms of $K$ are related to those of $L$.

\begin{Deff}\label{D_K_pi}
Let $\Dc_K^\tau$ be the finest usc decomposition of $K$ into subcontinua that split none of the images $\tau(\delta)$ with $\delta\in\Dc_L^{PC}$.
\end{Deff}

\begin{conj}\label{conjecture}
$\Dc_K^{PC}=\Dc_K^\tau$.
\end{conj}



Hereafter in this section, we focus on the special case that $\tau:\Pc\rightarrow\Sc$ allows a special polygon diagram $\Pc$ such that all the vertices of $\Pc$ are sent by $\tau$ to the same point. When $\Sc$ is the torus, such a diagram is given in Figure \ref{torus_tau}.  \begin{figure}[ht]
\includegraphics[trim = 0 0 0 0, width=0.3\textwidth]{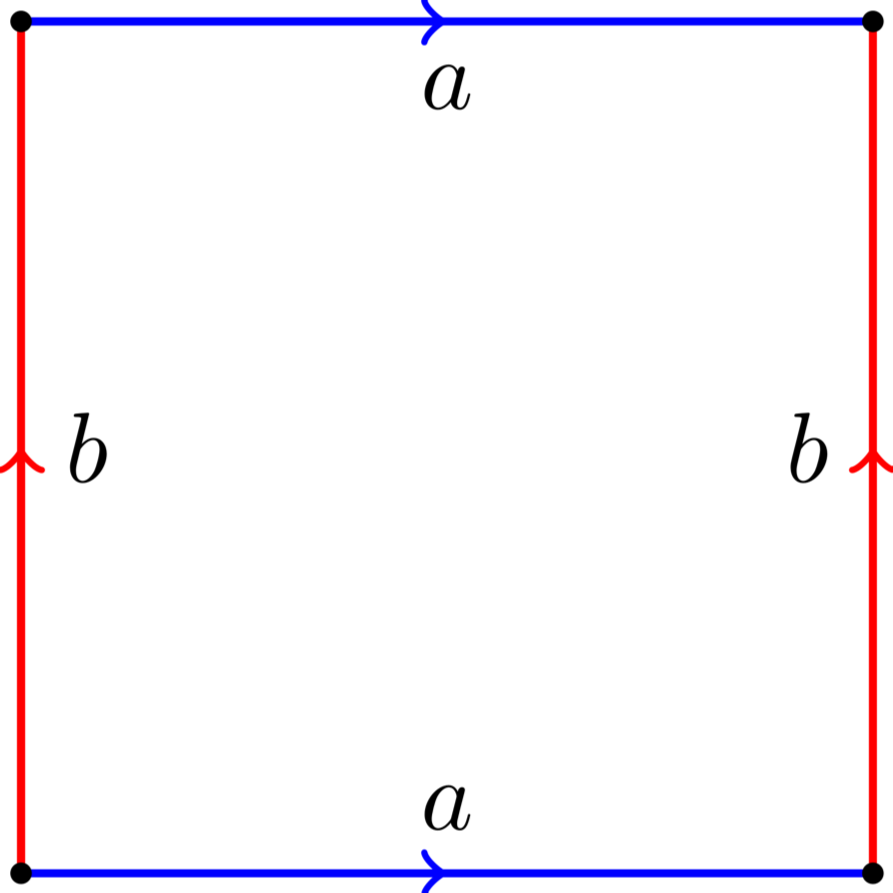}
\caption{The identification rule for  $\tau$.}\label{torus_tau}
\end{figure}

In the sequel, we fix a parametrization $\bbT=\left\{\left(e^{2\pi t\textbf{i}},e^{2\pi s\textbf{i}}\right): 0\le t,s\le1\right\}$. We also set $\varphi(x+y\textbf{i})=\left(e^{2\pi x\textbf{i}},e^{2\pi y\textbf{i}}\right)$ and call it the {\it covering map} of $\bbC$ onto $\bbT^2$. Clearly,  $\tau:[0,1]^2\rightarrow\bbT$ is just the restriction of $\varphi$ to $[0,1]^2$.   In Example \ref{torus_1}, we consider a continuum $K\subset\bbT$ and determine all its atoms.

\begin{exam}\label{torus_1} Let
$K=\left\{\tau(t+s\textbf{i}): \{t,s\}\cap\Kc\ne\emptyset\right\}$ where $\Kc$ is Cantor's ternary set.
Let $\Kc_0=\Kc\setminus(\{a_i:i\ge1\}\cup\{b_i:i\ge1\})$ where $(a_i,b_i)\ (i\ge1)$ are the components of $[0,1]\setminus\Kc$. Then an atom of $K$ falls into one of three possibilities.
\begin{itemize}
\item[(1)] The boundary of a Jordan domain which is  a component of $\bbT^2\setminus K$. 
\item[(2)] A line segment of the form $\{\tau(t+s\textbf{i}): \ s\in [a_j,b_j]\}$ or $\{\tau(s+t\textbf{i}): \ s\in [a_j,b_j]\}$
with  $(a_j,b_j)$ a component  of $[0,1]\setminus\Kc$ and $t\in \Kc_0$.
\item[(3)] A singleton  $\{\tau(t+s\textbf{i})\}$ with $t,s\in\Kc_0$. 
\end{itemize}
For an illustration, see Figure~\ref{K_on_torus}.

\begin{figure}[ht]
\includegraphics[trim = -50 0 0 0, width=0.8\textwidth]{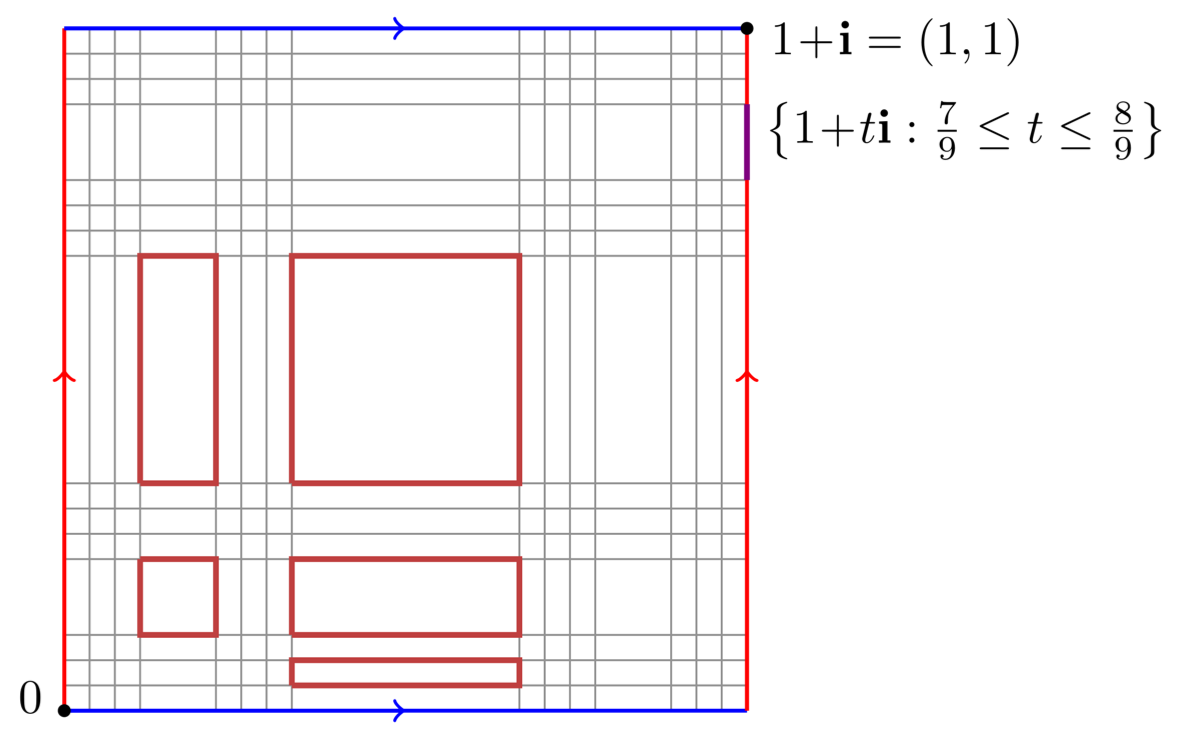}
\caption{A simple illustration of  $\tau^{-1}(\delta)$ for some atoms $\delta\in\Dc_K^{PC}$.}\label{K_on_torus}
\end{figure}
\end{exam}


As an illustration of Theorem \ref{invariance}, Example \ref{torus_2} concerns a covering map $f:\bbT^2\rightarrow\bbT^2$ and shows how the atoms of $f^{-1}(K)$ are connected to those of $K$.

\begin{exam}\label{torus_2}
Let $K$ and $L$ be defined as in Example \ref{torus_1}.
Given $m,n\in\bbZ\setminus\{0\}$, let $f\Big(\left(e^{2\pi{\mathbf i}t}, e^{2\pi{\mathbf i}s}\right)\Big)=\left(e^{2\pi{\mathbf i}mt}, e^{2\pi{\mathbf i}ns}\right)$ for $0\le t,s\le 1$. Then $f$ is a covering map of $\bbT^2$ onto itself, whose lift  via  $\varphi:\bbC\rightarrow\bbC$ is just $\widetilde{f}(x+y\textbf{i})=mx+ny\textbf{i}$.
Set $K_1=f^{-1}(K)$ and  $L_1=\tau^{-1}\left(K_1\right)$.
Then $L_1=\bigcup_{u,v}L+u+v\textbf{i}$ with $u$ running through $\{0,\ldots,m-1\}$ and $v$ running through $\{0,\ldots,n-1\}$ is a continuum. One can check that every atom of $K_1$ is
sent onto an atom of $K$ by the covering map $f:\bbT^2\rightarrow\bbT^2$. See Figure \ref{torus_1(2)}  for an illustration of $L_1$, with $m=3$ and $n=2$.
\begin{figure}[ht]
\vskip -0.2cm
\includegraphics[trim = 0 5 0 5, width=0.5\textwidth]{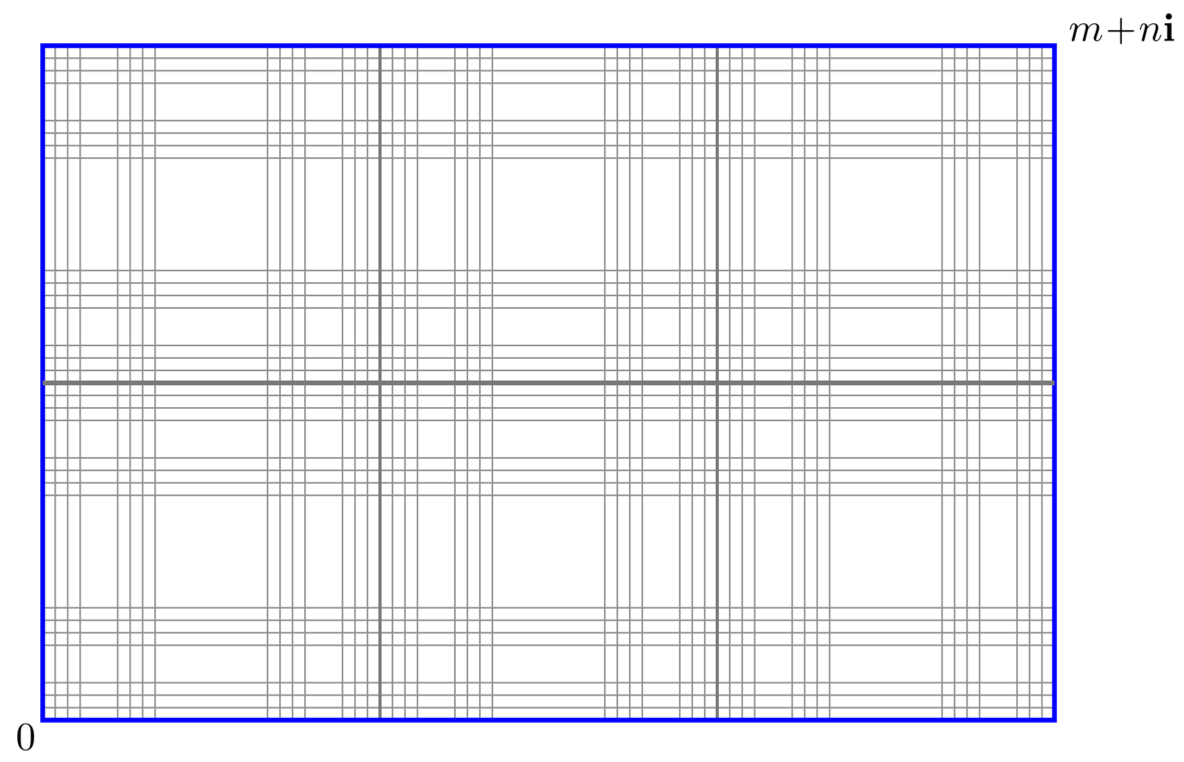}
\caption{The translates $L_1+u+v\textbf{i}$  $(0\le u\le m-1, 0\le v\le  n-1)$.}\label{torus_1(2)}
\vskip -0.2cm
\end{figure}
\end{exam}


Notice that the two compacta $K,K_1\subset\bbT^2$ share a fundamental property: every component $U$ of the complement is a Jordan domain whose boundary is an atom of $K$.
Therefore, we may extend $\Dc_K^{PC}$ (respectively, $\Dc_{K_1}^{PC}$) to an usc decomposition of $\bbT^2$ into subcontinua, each of which has a connected complement in $\bbT^2$. Each element of the extension is either an atom of $K$ (respectively, of $K_1$) or the closure $\overline{U}$, where $U$ is a component of $\bbT\setminus K$ (respectively, of $\bbT\setminus K_1$). By a generalization of Moore's theorem to surfaces (see {\em e.g.} \cite[Section~5, Theorem~6 and Section~25, Theorem~1]{Daverman86}) the resulting quotient space is homeomorphic to $\bbT^2$.

\begin{rema}
For closed surfaces $\Sc$ of higher genus, we may construct special continua $K\subset\Sc$ with the above mentioned property. For an oriented surface $\Sc$ of genus two, such a continuum $K$ is presented in the following Figure \ref{torus_3}, which describes the identification rule of a projection $\tau$ from an octagon $\Pc$ onto $\Sc$.
\begin{figure}[ht]
\begin{tabular}{cc}
\includegraphics[trim = 10 10 10 10, width=0.3\textwidth]{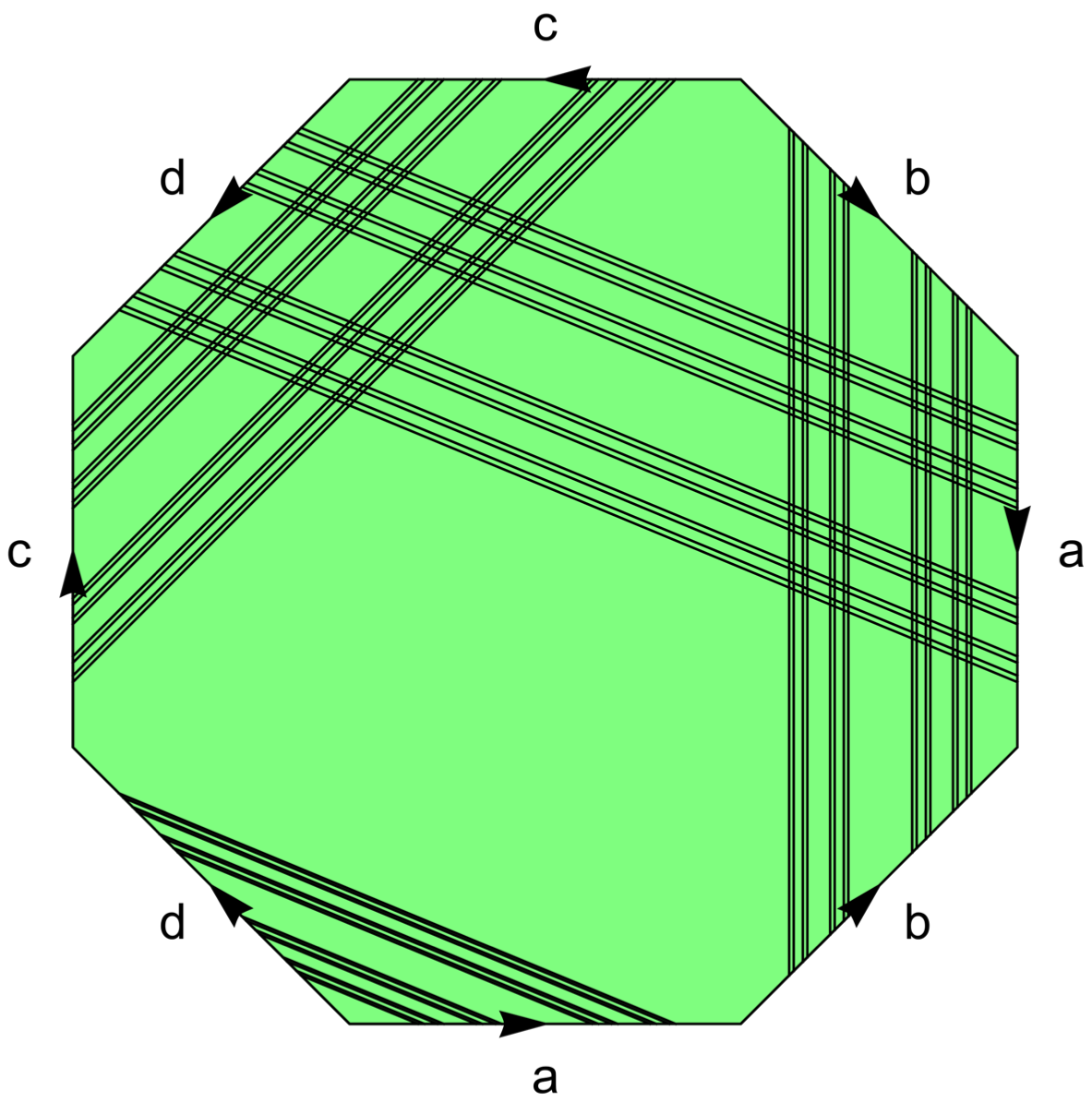}
&
\includegraphics[trim = 50 100 30 100, width=0.45\textwidth]{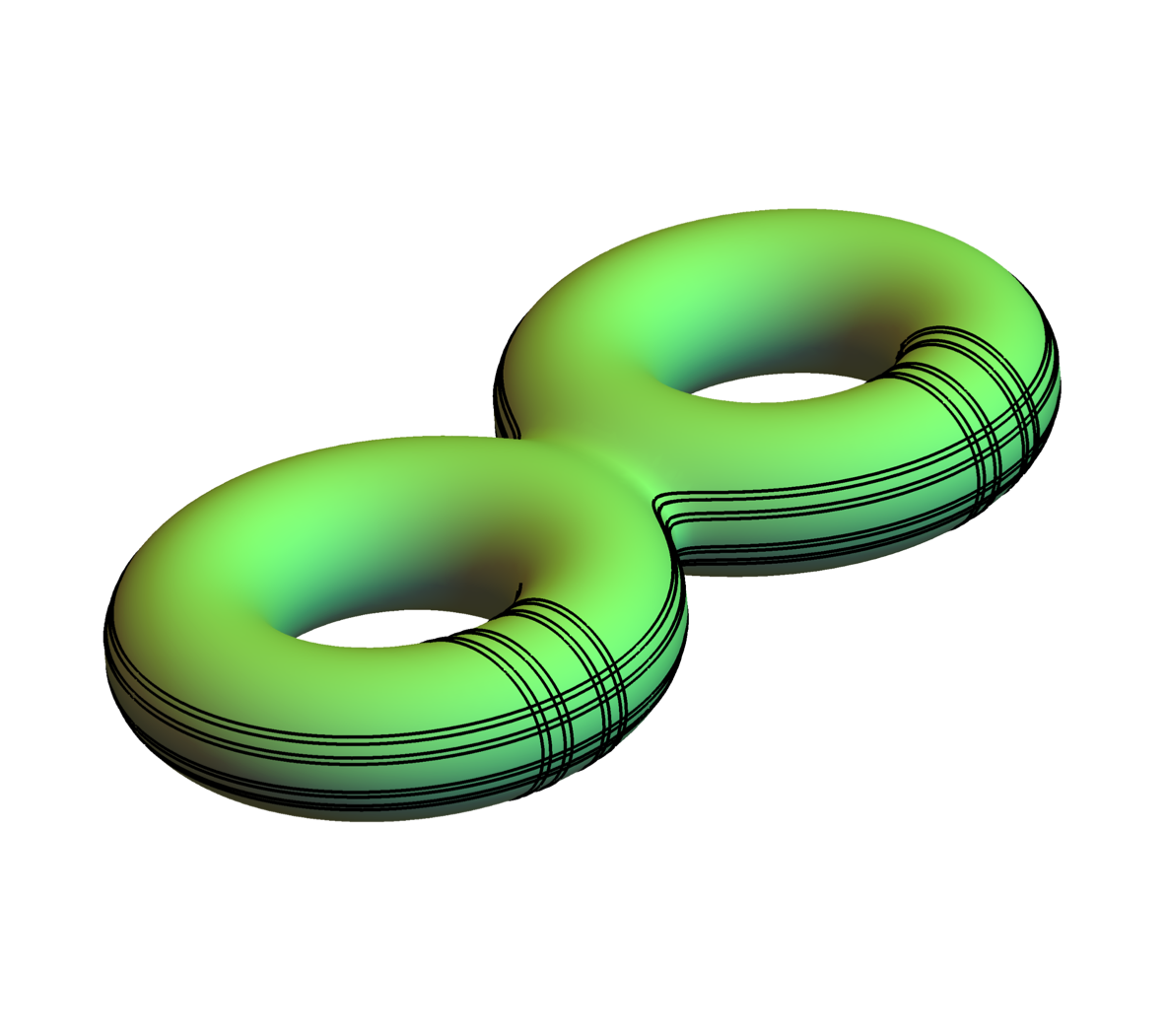}
\end{tabular}
\caption{The identification rule of $\tau$ and $\tau^{-1}(K)\subset\Pc$, $K\subset\Sc$.}\label{torus_3}
\vskip -0.2cm
\end{figure}
Further continua on closed surfaces of higher genus may be constructed in a similar way. The same construction still works, if we turn to closed surfaces that are nonorientable.
\end{rema}

\section{Atomic decompositions of compacta in $\bbR^3$}
\label{sec:ex2}

In this section we show that our theory cannot be extended to higher dimensional Euclidean space. In particular, we analyze a continuum $K$ (and a variant of it), having its origin in \cite[Example 7.1]{LLY-2019}. This continuum is not locally connected and satisfies $\widetilde{R_K}=\{(x,x): x\in K\}$. Throughout this section, we will use the projections  $p_1(x_1,x_2,x_3)=(x_1,x_2)$ and $p_2(x_1,x_2,x_3)=(x_1,x_3)$ and the abbreviations  $I_0=\{0\}\times[0,1]\times\{0\}$ and $\Lambda_0=\{0\}\times[0,1]^2$. 

\begin{exam}\label{Example_no_CD}
Let $K\subset\bbR^3$ defined by
\[
K=\bigg(\Big(\{0\}\cup\Big\{\frac1n: n\in\bbN\Big\}\Big)\times[0,1]^2\bigg) \cup \big([0,1]\times\{0\}\times[0,1]\big) \cup \big([0,1]^2\times\{0\}\big). 
\]
Then $\Dc_i^*=\left\{p_i^{-1}(u): u\in[0,1]^2\right\}$ for $i=1,2$ is a Peano decomposition, while the only usc decomposition that refines both, $\Dc_1^*$ and $\Dc_2^*$, is given by the trivial decomposition $\{\{u\}: u\in K\}$. Therefore, the core decomposition does not exist. \end{exam}

We will determine all the atomic decompositions of a continuum that is a slight variant of $K$. It is given by
\[
K_0=\bigg(\Big(\{0\}\cup\Big\{\frac1n: n\in\bbN\Big\}\Big)\times[0,1]^2\bigg) \cup \bigg([0,1]^2\times\{0\}\bigg).
\]
\begin{figure}[ht]
\begin{center}
\includegraphics[trim = 0 10 0 5, width=0.3\textwidth]{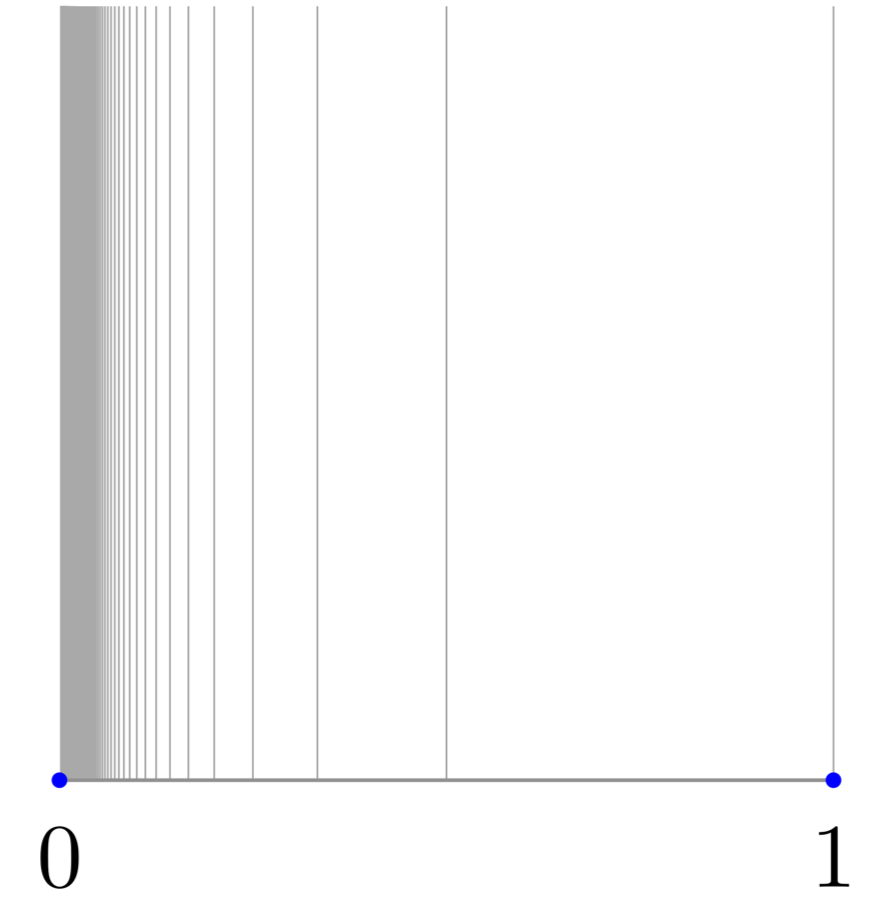}
\end{center}
\caption{An illustration of $p_2(K_0)\subset\bbR^2$.}\label{K_projection}
\end{figure}

Let $\mathfrak{M}_{K_0}$ consist of the Peano decompositions of $K_0$. We address three issues.
The first concerns how to construct an atomic decomposition $\Dc_1\in\mathfrak{M}_{K_0}$ and a family of homeomorphisms $h_\lambda:K_0\rightarrow K_0$ depending on a parameter $\lambda\in(0,1)$ such that by applying $h_\lambda$ to $\Dc_1$ we can obtain uncountably many atomic decompositions $h_\lambda(\Dc_1):=\{h_\lambda(\delta): \delta\in\Dc_1\}$.

\begin{exam}\label{continuation-1} Let
$\Dc_1$ consist of all the line segments  $\{0\}\times\{x_2\}\times[0,1]$ with $x_2\in[0,1]$ and all the singletons $\{(x_1,x_2,x_3)\}$ with  $x_1>0$. Then $\Dc_1\in\mathfrak{M}_{K_0}$ and the resulting quotient space is homeomorphic to the union 
\begin{equation}\label{quotient-1}
X_1:=\big([0,1]^2\times\{0\}\big) \cup \bigg(\bigcup_{n\in\bbN}
\Big\{\frac1n\Big\}\times[0,1]\times\Big[0,\frac1n\Big]\bigg).
\end{equation}
See Figure \ref{quotient_projection} for a projection of $X_1$ under $p_2$.
\begin{figure}[h]
\begin{subfigure}[b]{0.45\textwidth}
         \centering
     \includegraphics[trim = 0 5 0 5, width=0.75\textwidth]{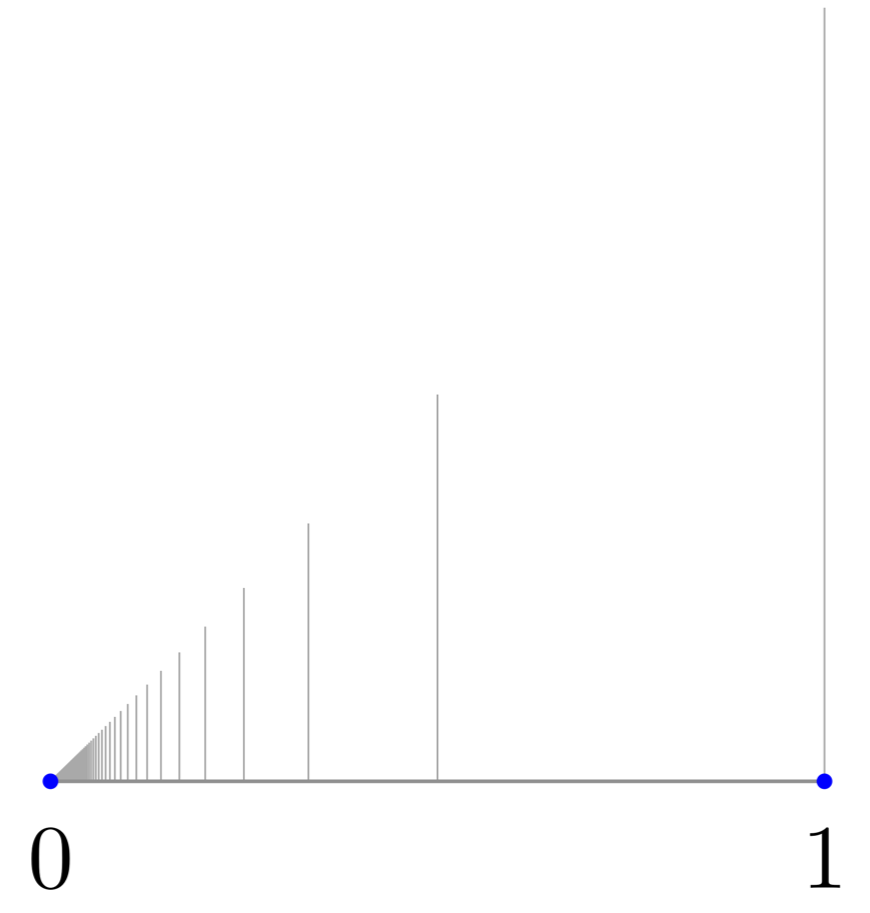}
     \caption{An illustration of $p_2(X_1)\subset\bbR^2$.}\label{quotient_projection}
\end{subfigure}
\hspace{0.618cm}
\begin{subfigure}[b]{0.45\textwidth}
         \centering
     \includegraphics[trim = 0 5 0 5, width=0.75\textwidth]{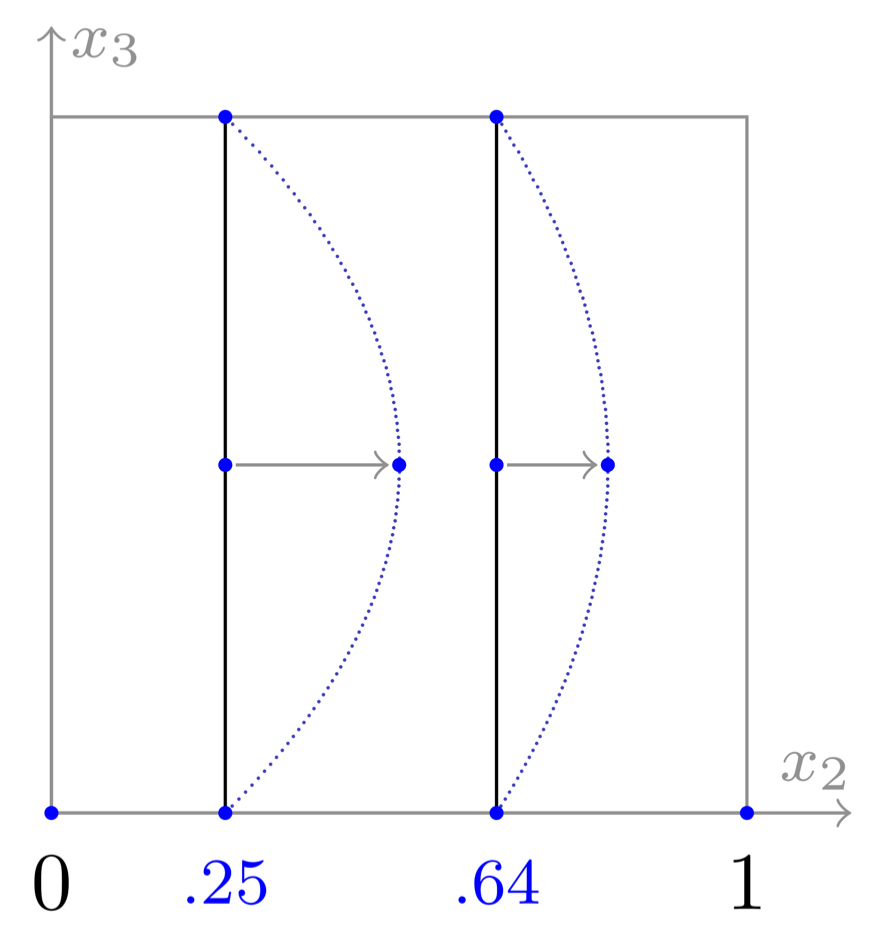}
\caption{An illustration of $\{x_1\}\times\{x_2\}\times[0,1]$ and its image under $h_\frac12$.}\label{fig:h_lambda}
\end{subfigure}
\end{figure}
Moreover, for $\lambda\in(0,1)$  we set \begin{equation}\label{h_lambda}
h_\lambda\left(\left[\begin{array}{c}
x_1\\ x_2\\ x_3\end{array}\right]\right)=\left[\begin{array}{c}
x_1\\ x_2\\ x_3\end{array}\right]+4x_3(1-x_3)(x_2^\lambda-x_2)\left[\begin{array}{c}
0\\ 1\\ 0\end{array}\right].
\end{equation}
See Figure \ref{fig:h_lambda} for an illustration of how the line segments $\{x_1\}\times\{x_2\}\times[0,1]$ with $\displaystyle x_1\in\{0\}\cup\Big\{\frac1n:n\in\bbN\Big\}$ and $x_2\in(0,1)$ are bent under $h_\lambda$, with $\lambda=\frac12$.
Note that every decomposition $h_\lambda(\Dc_1)=\{h_\lambda(\delta): \delta\in\Dc_1\}$ belongs to $\mathfrak{M}_{K_0}$. Also note that $\Dc_1$ is an atomic decomposition if and only if $h_\lambda(\Dc_1)$ is. 
\end{exam}

\begin{prop}\label{D1_atomic}
$\Dc_1$   is an atomic decomposition. Hence, every $h_\lambda(\Dc_1)$ is an atomic decomposition of $K_0$.

\end{prop}
\begin{proof}Assume on the contrary that $\Dc_1$ were not atomic, so that we could find a decomposition $\Dc\in\mathfrak{M}_{K_0}$ that refines $\Dc_1$, such that there is at least one element $\delta\in\Dc_1$ of the form $\{0\}\times\{s\}\times[1-t,1]$, with $s,t\in(0,1)$. Let $\pi_\Dc:K_0\rightarrow\Dc$ be the corresponding quotient map and $\rho$ a metric on the quotient space that is compatible with the quotient topology. Then the distance under $\rho$ between the point $\pi_\Dc(\delta)$ and the continuum $\pi_\Dc(I_0)$ is a positive number, say $\varepsilon_0$.  There are two observations. First, the sequence
\[\displaystyle \left\{\pi_\Dc\left(\left\{\frac1n\right\}\times\{s\}\times\{1\}\right): \ n\ge1\right\}
\]
converges to $\pi_\Dc(\delta)$ in the quotient space. Second,  every subcontinuum of the quotient space that connects $\pi_\Dc(\delta)$ to a point in the above sequence must intersect the continuum $\pi_\Dc(I_0)$. Such a subcontinuum is of diameter $\ge \varepsilon_0$. It follows that the quotient space of $\Dc$ is not locally connected at the point $\pi_\Dc(\delta)$. This is absurd, since $\Dc\in\mathfrak{M}_{K_0}$. We are done.
\end{proof}

\begin{rema}\label{atomic_condition}
Let $\Dc\in\mathfrak{M}_{K_0}$. By slightly adjusting the above arguments, one may infer that every element $\delta\in\Dc$  intersecting  $\Lambda_0$ must intersect  $I_0$.
\end{rema}

The second issue is to find two  atomic decompositions $\Dc_2\ne\Dc_3$, neither of which can be obtained by applying a homeomorphism  $h:K_0\rightarrow K_0$ to $\Dc_1$. The quotient space of $\Dc_2$ is still homeomorphic with that of $\Dc_1$, while that of $\Dc_3$ is not.

\begin{exam}\label{continuation-2} Let
$\Dc_2$ be a decomposition with two conditions:  (1) every singleton $\{(x_1,x_2,x_3)\}$ with  $x_1>0$ belongs to $\Dc$; (2) all elements $\delta\subset\Lambda_0$ form an usc decomposition of $\Lambda_0$ and they are given in the following way. One of them is
\[
\displaystyle\delta_0=\Big\{(0,x_2,x_3)\in K: x_2=\frac12\ \text{or}\ \Big(x_2-\frac12\Big)^2+\Big(x_3-\frac12\Big)^2\le\frac{1}{16}\Big\}.
\]
The others are each an arc connecting  $(0,x_2,0)$ to $(0,x_2,1)$ for some $x_2\ne\frac12$.
Then $\Dc_2$ is an atomic decomposition of $K_0$ and the resulting quotient space is homeomorphic to $X_1$ given in Equation \ref{quotient-1}.
See Figure~\ref{D2andD3} (left) for an illustration of $Q_0:=\{(x_2,x_3): (0,x_2,x_3)\in\delta_0\}$.
\begin{figure}[ht]
\begin{tabular}{cc}
\includegraphics[trim = 0 10 0 5, width=0.33\textwidth]{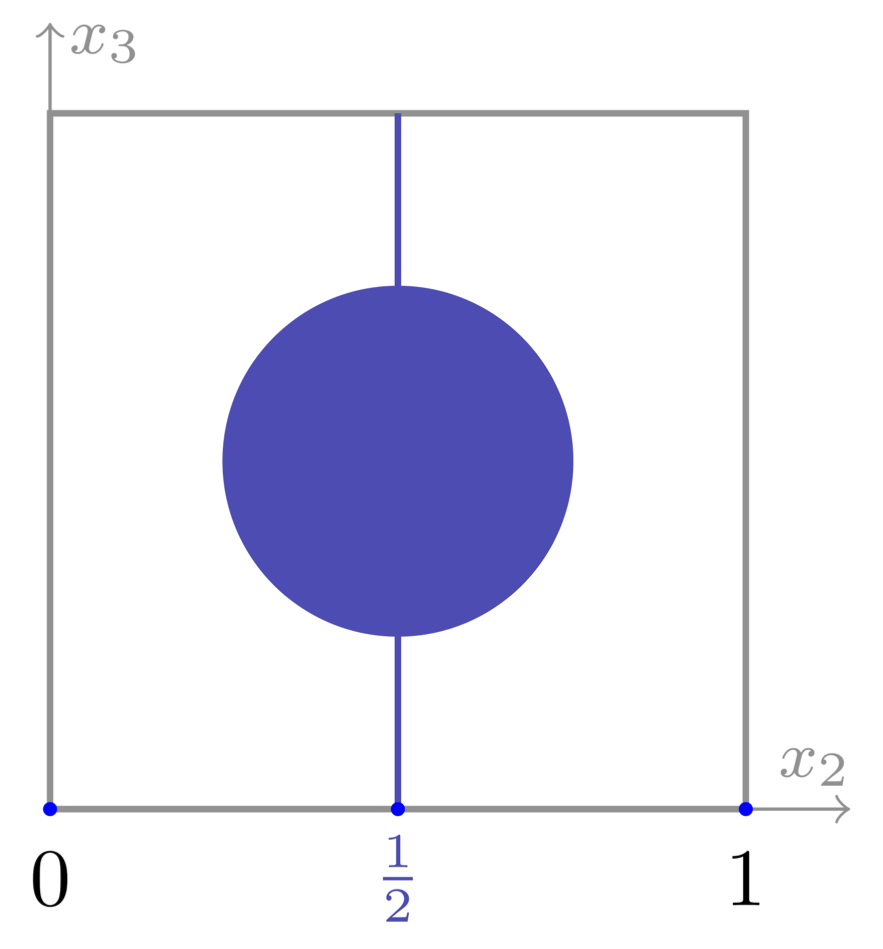}
&\hspace{1cm}
\includegraphics[trim = 0 10 0 5, width=0.33\textwidth]{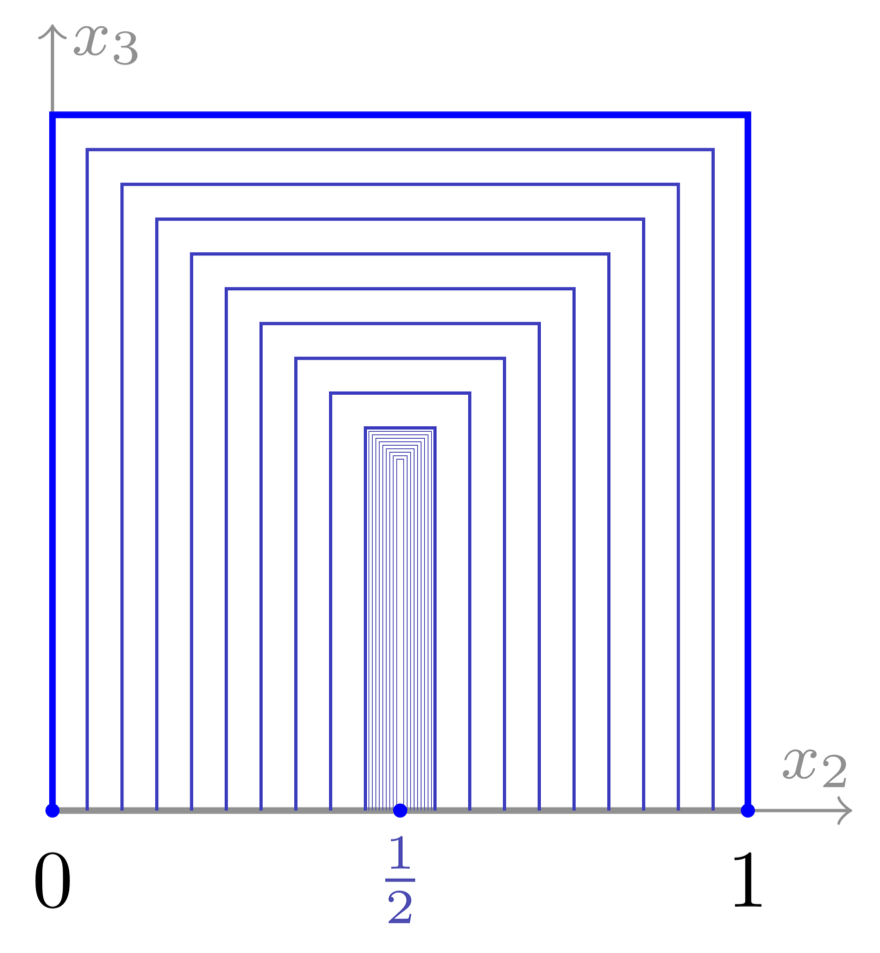}
\end{tabular}
\caption{An illustration of $Q_0$(left) and some elements of $\Dc_3$ (right).}\label{D2andD3}
\vskip -0.2cm
\end{figure}
The decomposition of $\Dc_3$ shares with $\Dc_2$ all its elements that are singletons.
Moreover, its nondegenerate elements are of the form
\[
(\{0\}\times\{x_2,1-x_2\}\times[0,1-x_2]) \cup (\{0\}\times[x_2,1-x_2]\times\{1-x_2\})
\] 
with $0\le x_2\le \frac12$.
See Figure \ref{D2andD3}~(right) for an illustration of some nondegenerate elements of $\Dc_3$. Note that the quotient space of $\Dc_3$ is homeomorphic with a special quotient space of $X_1$, by identifying every point $(0,x_2,0)$ with $(0,1-x_2,0)$ with $0\le x_2<\frac12$. It is routine to verify that the quotient space of $\Dc_2$ is homeomorphic to that  of $\Dc_1$, while that of $\Dc_3$ is  not.
\end{exam}

Finally, the third issue is to  identify all atomic decompositions of $K_0$. To gain further insights, let us start from some basic observations concerning what a generic atomic decomposition of $K_0$ may look like.
Fix a Peano decomopsition $\Dc\in\mathfrak{M}_{K_0}$. By chopping up certain elements of $\Dc$ into smaller pieces, we may obtain a finer decomposition $\Dc^*$,  still belonging to $\mathfrak{M}_{K_0}$, that turns out to be atomic in the end. Every  $\delta\in\Dc$ disjoint from  $\Lambda_0$ will be divided into singletons.  Every $\delta\in\Dc$ that intersects  $\Lambda_0$ will be further divided into a collection $\Ec_\delta$ of subcontinua that form an usc decomposition of $\delta$. There are two types of elements in $\Ec_\delta$. The first are the singletons $\{u\}$ with $u\in\delta\setminus\Lambda_0$. The second are components of  $\delta\cap\Lambda_0$. By the Cut Wire Theorem (see Lemma \ref{lem:CWT}), each of those components intersects  $I_0$. Now it is routine to verify that the quotient space of $\Dc^*$ is a Peano continuum.

We summarize all the atomic decompositions of $K_0$ in the following.

\begin{exam}\label{continuation-3}
For the above-mentioned continuum $K_0$, every atomic decomposition $\Dc$ (with Peano quotient) is a Peano decomposition whose elements $\delta$ satisfy the following  conditions.
\begin{itemize}
\item[(1)] If $\delta\setminus\Lambda_0\ne\emptyset$ then $\delta=\{u\}$ for some $u\in K\setminus\Lambda_0$.
\item[(2)] If  $\delta\subset\Lambda_0$ then $\delta\cap I_0\ne\emptyset$. 
\item[(3)] All those $\delta\in\Dc$ with $\delta\subset\Lambda_0$  form an usc decomposition $\Dc_0$ of $\Lambda_0$. 
\end{itemize}
It follows that the resulting quotient space of $\Dc_0$ is a dendrite, which is a subset of the quotient space of $\Dc$.
\end{exam}

To conclude this section, let us look closer at the atomic decompositions of $K_0$. Condition (1) from  Example \ref{continuation-2} implies that for any atomic decomposition $\Dc$  there is an embedding $h_\Dc$ of  $X_0:=X_1\setminus I_0$ into the quotient space of $\Dc$ such that the image $h_\Dc(X_0)$ is a dense subset. 
Here $X_1$ is given as in Equation \ref{quotient-1}. 
Condition (2) implies that
$\Dc$ restricted to  $I_0$ is an usc decomposition. By identifying $I_0$ with $[0,1]$, we may associate to $\Dc$ a closed equivalence of $[0,1]$, to be denoted by $\sim_\Dc$.
There are two possibilities, either $0$ and $1$ are contained in the same class or they are not. In the former case, we may associate to $\sim_\Dc$ an usc decomposition of $\overline{\bbD}$ into convex sets, with respect to hyperbolic metric. Such a decomposition has its origin in the study of polynomial Julia sets and is often called a lamination. See for instance \cite{BL92,BCO11,Kiwi04,Thurston09}. Therefore, the topological models for certain polynomial Julia sets, and for the Mandelbrot set or its higher degree analogues, may be realized as a subset of the quotient space of an appropriate atomic decomposition $\Dc$ of $K_0$.


\bibliographystyle{abbrv}
\bibliography{biblio}

\end{document}